\documentclass[a4paper,11pt]{amsart}
\usepackage{amssymb,amsmath}
\usepackage{amsopn}
\usepackage{mathrsfs}
\usepackage{color}
\usepackage{mathtools}
\usepackage{wasysym}
\usepackage{vmargin}
\usepackage{pdfsync}
\usepackage{vmargin}
\usepackage{color}
\usepackage{amscd}
\usepackage{verbatim}
\usepackage{bbm}
\usepackage{tikz}
\usepackage{hyperref}
\usetikzlibrary{patterns}

\newtheorem{theorem}{Theorem}[section]
\newtheorem{definition}[theorem]{Definition}
\newtheorem{lemma}[theorem]{Lemma}
\newtheorem{proposition}[theorem]{Proposition}
\newtheorem{corollary}[theorem]{Corollary}
\newtheorem{remark}[theorem]{Remark}
\newtheorem{example}[theorem]{Example}
\numberwithin{equation}{section}

\newcommand{\rr}{\mathbb{R}}
\newcommand{\eps}{\varepsilon}
\newcommand{\nn}{\mathbb{N}}
\newcommand{\cc}{\mathbb{C}}
\def\un{{\mathrm{1~\hspace{-1.4ex}l}}}

\def\diam{\operatorname{diam}}

\def\N{\mathbb N}

\def\R{\mathbb R}

\def\val#1{\vert#1\vert}

\def\l2{L^2(\R^{n})}
\def\L2{L^2(\R^{2n})}
\def\supp{\operatorname{supp}}

\def\eps{\varepsilon}

\def\mat22#1#2#3#4{\begin{pmatrix}#1&#2\\ #3&#4\end{pmatrix}}

\begin{document}
\title[Uncertainty principles and null-controllability]{Uncertainty principles in Gelfand-Shilov spaces and null-controllability}

\author{Jérémy \textsc{Martin}}

\address{\noindent \textsc{Jérémy Martin, Univ Rennes, CNRS, IRMAR - UMR 6625, F-35000 Rennes, France
}}
\email{jeremy.martin@univ-rennes1.fr}

\makeatletter
	\@namedef{subjclassname@2020}{\textup{2020} Mathematics Subject Classification}
\makeatother

\begin{abstract}
We provide new uncertainty principles for functions in a general class of Gelfand-Shilov spaces. These results apply, in particular, with the classical Gelfand-Shilov spaces as well as for spaces of functions with weighted Hermite expansions. Thanks to these uncertainty principles, we derive null-controllability results for evolution equations with adjoint systems enjoying smoothing effects in specific Gelfand-Shilov spaces. More precisely, we consider control subsets which are thick with respect to a quasi linearly growing density and establish sufficient conditions on the growth of the density to ensure null-controllability of these evolution equations.

\end{abstract}

\maketitle
\section{Introduction}
This paper aims at broadening the understanding of the link between uncertainty principles and localized controllability of evolution equations. An uncertainty principle is a property which gives some limitations on the simultaneous concentration of a function and its Fourier transform. There exist different forms of uncertainty principles and one of them consists in studying the support of functions whose Fourier transforms are localized. The Logvinenko-Sereda Theorem \cite{logvinenko-sereda} ensures the equivalence of the norms $\| \cdot \|_{L^2(\rr^d)}$ and $\| \cdot \|_{L^2(\omega)}$, where $\omega \subset \rr^d$ is a measurable subset, on the subspace
$$\big\{f \in L^2(\rr^d); \ \supp \hat{f} \subset \overline{B(0,R)} \big\} \quad \text{with} \quad R>0,$$ where $\hat{f}$ denotes the Fourier transform of $f$,
as soon as $\omega$ is thick. The thickness property is defined as follows:
\begin{definition}\label{thick_def}
Let $d \in \nn^*$ and $\omega$ be a measurable subset of $\rr^d$. For $0<\gamma \leq 1$ and $L>0$, the set $\omega$ is said to be $\gamma$-thick at scale $L>0$ if and only if 
\begin{equation*}
\forall x \in \rr^d, \quad |\omega \cap (x+[0,L]^d)| \geq \gamma L^d,
\end{equation*}
where $|A|$ denotes the Lebesgue measure of the measurable set $A$. 
The set $\omega$ is said to be thick if and only if $\omega$ is $\gamma$-thick at scale $L>0$ for some $0<\gamma \leq 1$ and $L>0$.
\end{definition}
We define more generally the thickness with respect to a density:
\begin{definition}\label{thick_density}
Let $d \in \nn^*$, $0<\gamma\leq 1$, $\omega$ be a measurable subset of $\rr^d$ and $\rho : \rr^d \longrightarrow (0,+\infty)$ a positive function. The set $\omega$ is said to be $\gamma$-thick with respect to $\rho$ if and only if
\begin{equation*}
\forall x \in \rr^d, \quad |\omega \cap B(x,\rho(x))| \geq \gamma |B(x,\rho(x))|,
\end{equation*}
where  $B(x,L)$ denotes the Euclidean ball of $\rr^d$ centered at $x$ with radius $L$. 
\end{definition}
Of course, a measurable subset of $\rr^d$ is thick if and only if it is thick with respect to a positive constant density.
Kovrijkine provided a quantitative version of the Logvinenko-Sereda Theorem in \cite[Theorem~3]{Kovrijkine}:
\begin{theorem}[Kovrijkine {\cite[Theorem~3]{Kovrijkine}}]\label{Kovrijkine1} Let $\omega \subset \rr^d$ be a measurable subset $\gamma$-thick at scale $L>0$.
There exists a universal positive constant $C>0$ independent on the dimension $d \geq 1$ such that for all $f \in L^2(\rr^d)$ satisfying $\supp \hat{f} \subset J$, with $J$ a cube with sides of length $b$ parallel to coordinate axes, 
\begin{equation}\label{kovrijkine1.1}
\|f\|_{L^2(\rr^d)} \leq c(\gamma,d, L, b) \|f\|_{L^2(\omega)},
\end{equation} 
with $$c(\gamma, d, L, b)= \Big( \frac{C^d}{\gamma} \Big)^{Cd(Lb+1)}.$$
\end{theorem}
The thickness property was recently shown to play a key role in spectral inequalities for finite combinations of Hermite functions. In \cite[Theorem~2.1]{MP}, the authors establish quantitative estimates with an explicit dependence on the energy level $N$ with respect to the growth of the density appearing in Definition~\ref{thick_density}:

\medskip

\begin{theorem}[Pravda-Starov \& Martin]\label{Spectral}
Let $\rho : \rr^d \longrightarrow (0,+\infty)$ be a $\frac{1}{2}$-Lipschitz positive function with $\rr^d$ being equipped with the Euclidean norm,
such that there exist some positive constants $0< \eps \leq 1$, $m>0$, $R>0$ such that
\begin{equation*}
\forall x \in \rr^d, \quad 0<m \leq \rho(x) \leq R{\left\langle x\right\rangle}^{1-\eps}.
\end{equation*}
Let $\omega$ be a measurable subset of $\rr^d$ which is $\gamma$-thick with respect to the density $\rho$.
Then, there exist some positive constant $\kappa_d(m, R, \gamma, \eps)>0$, $\tilde{C}_d(\eps, R) >0$ and a positive universal constant $\tilde{\kappa}_d >0$ such that
\begin{equation}\label{spec_ineq}
\forall N \geq 1, \ \forall f \in \mathcal{E}_N, \quad \|f\|_{L^2(\rr^d)} \leq \kappa_d(m, R, \gamma, \eps) \Big( \frac{\tilde{\kappa}_d}{\gamma} \Big)^{\tilde{C}_d(\eps, R) N^{1-\frac{\eps}{2}}} \|f\|_{L^2(\omega)},
\end{equation}
with $\mathcal E_{N}$ being the finite dimensional vector space spanned by the Hermite functions $(\Phi_{\alpha})_{\val \alpha \leq N}$.
\end{theorem}

\medskip
We refer the reader to Section~\ref{Hermite_functions} for the definition and some notations related to Hermite functions $(\Phi_{\alpha})_{\alpha \in \nn^d}$. 
We emphasize that Theorem~\ref{Spectral} ensures, in particular, the equivalence of the norms $\| \cdot \|_{L^2(\rr^d)}$ and $\| \cdot \|_{L^2(\omega)}$ on the subspace $\mathcal{E}_N$ as soon as the measurable subspace $\omega$ is thick with respect to a suitable density. Actually, contrary to the case when the functional subspace is the space of functions whose Fourier transforms are compactly supported, this fact holds true as soon as $\omega$ is a measurable subset of positive measure. As explained by the authors of \cite[Section~2]{kkj}, the analyticity property of finite combinations of Hermite functions together with an argument of finite dimension imply that for all $N \in \nn$, there exists a positive constant $C_N(\omega)>0$ such that $$\forall f \in \mathcal{E}_N, \quad \|f\|_{L^2(\rr^d)} \leq C_N(\omega) \|f \|_{L^2(\omega)},$$
as soon as $|\omega| >0$. The main interest of Theorem~\ref{Spectral} is the quantitative estimate from above on the growth of the positive constant $C_N(\omega)$ with respect to the energy level $N$, which is explicitly related to the growth of the density $\rho$ thanks to $\eps$. As the norms $\| \cdot \|_{L^2(\rr^d)}$ and $\| \cdot \|_{L^2(\omega)}$ are not equivalent on $L^2(\rr^d)$ when $|\rr^d \setminus \omega|>0$, the constant $C_N(\omega)$ does have to blow up when $N$ tends to infinity. However, the asymptotic of this blow-up is very much related to the geometric properties of the control set $\omega$, and understanding this asymptotic can be assessed as an uncertainty principle.


One of the purpose of this work is to establish new uncertainty principles holding in a general class of Gelfand-Shilov spaces and to provide sufficient conditions on the growth of the density allowing these uncertainty principles to hold. Furthermore, this paper aims at providing new null-controllability results as a byproduct of these uncertainty principles.
Indeed, some recent works have highlighted the key link between uncertainty principles and localized control of evolution equations matters. Thanks to the explicit dependence of the constant with respect to the length of the sides of the cube in \eqref{kovrijkine1.1}, Egidi and Veseli\'c~ \cite{veselic}; and Whang, Whang, Zhang and Zhang \cite{Wang} have independently established that the heat equation  
\begin{equation*}\label{heat}
\left\lbrace \begin{array}{ll}
(\partial_t -\Delta_x)f(t,x)=u(t,x)\un_{\omega}(x)\,, \quad &  x \in \mathbb{R}^d,\ t>0, \\
f|_{t=0}=f_0 \in L^2(\rr^d),                                       &  
\end{array}\right.
\end{equation*}
is null-controllable in any positive time $T>0$ from a measurable control subset $\omega \subset \rr^d$ if and only if the control subset $\omega$ is thick in $\rr^d$. By using the same uncertainty principle, Alphonse and Bernier established in \cite{AlphonseBernier} that the thickness condition is necessary and sufficient for the null-controllability of fractional heat equations
\begin{equation}\label{fractional_heat}
\left\lbrace \begin{array}{ll}
(\partial_t + (-\Delta_x)^s)f(t,x)=u(t,x)\un_{\omega}(x)\,, \quad &  x \in \mathbb{R}^d,\ t>0, \\
f|_{t=0}=f_0 \in L^2(\rr^d),                                       &  
\end{array}\right.
\end{equation}
when $s >\frac{1}{2}$. On the other hand, Koenig showed in \cite[Theorem~3]{Koenig} and \cite[Theorem~2.3]{Koenig_thesis} that the null-controllability of \eqref{fractional_heat} fails from any non-dense measurable subset of $\rr$ when $0< s \leq \frac{1}{2}$. In \cite{AlphonseMartin}, Alphonse and the author point out the fact that the half heat equation, which is given by \eqref{fractional_heat} with $s= \frac{1}{2}$, turns out to be approximately null-controllable with uniform cost if and only if the control subset is thick. 
Regarding the spectral inequalities in Theorem~\ref{Spectral}, thanks to the quantitative estimates \eqref{spec_ineq}, Pravda-Starov and the author established in \cite[Corollary~2.6]{MP} that the fractional harmonic heat equation 
 \begin{equation*}
\left\lbrace \begin{array}{ll}
\partial_tf(t,x) + (-\Delta_x+ |x|^2)^s f(t,x)=u(t,x)\un_{\omega}(x), \quad &  x \in \mathbb{R}^d,\ t>0, \\
f|_{t=0}=f_0 \in L^2(\rr^d),                                       &  
\end{array}\right.
\end{equation*}
with $\frac{1}{2} < s \leq 1$, is null-controllable at any positive time from any measurable set $\omega$ which is thick with respect to the density 
\begin{equation*}
\forall x \in \rr^d, \quad \rho(x)= R \langle x \rangle^{\delta},
\end{equation*}
with $0 \leq \delta < 2s-1$ and $R>0$.
More generally, the result of \cite[Theorem~2.5]{MP} shows that this thickness condition is a sufficient condition for the null-controllability of a large class of evolution equations associated to a closed operator whose $L^2(\rr^d)$-adjoint generates a semigroup enjoying regularizing effects in specific symmetric Gelfand-Shilov spaces $S_{\frac{1}{2s}}^{\frac{1}{2s}}$.

The sufficiency of the thickness conditions for control subsets to ensure null-controllability results for these evolution equations is derived from an abstract observability result based on an adapted Lebeau-Robbiano method established by Beauchard and Pravda-Starov with some contributions of Miller in \cite[Theorem~2.1]{BeauchardPravdaStarov}. This abstract observability result was extended in~\cite[Theorem~3.2]{BEP} to the non-autonomous case with moving control supports under weaker dissipation estimates allowing a controlled blow-up for small times in the dissipation estimates.
The main limitation in the work \cite{MP} is that Hermite expansions can only characterize symmetric Gelfand-Shilov spaces (see Section~\ref{gelfand}) and therefore, the null-controllability results in \cite{MP} are limited to evolution equations enjoying only symmetric Gelfand-Shilov smoothing effects. This work partially adresses this matter by investigating the null-controllability of evolution equations associated to anharmonic oscillators, which are known to regularize in non-symmetric Gelfand-Shilov spaces. More generally, we establish null-controllability results for abstract evolution equations whose adjoint systems enjoy smoothing effects in non-symmetric Gelfand-Shilov spaces. This work precisely describes how the geometric properties of the control subset are related to the two indexes $\mu, \nu$ defining the Gelfand-Shilov space $S^{\mu}_{\nu}$.

This paper is organized as follows: 
In Section~\ref{uncertainty_principle_general}, new uncertainty principles and quantitative estimates are presented. We first establish uncertainty principles for a general class of Gelfand-Shilov spaces in Section~\ref{general_GS}. In a second time, we deal with the particular case of spaces of functions with weighted Hermite expansions in Section~\ref{up_symmetric_GS}. These results are derived from sharp estimates for quasi-analytic functions established by Nazarov, Sodin and Volberg in \cite{NSV}. Some facts and results related to quasi-analytic functions are recalled in Sections~\ref{main_results} and \ref{qa_section}.
Thanks to these new uncertainty principles, we establish sufficient geometric conditions for the null-controllability of evolutions equations with adjoint systems enjoying quantitative Gelfand-Shilov smoothing effects in Section~\ref{null_controllability_results}. 

 
\section{Statement of the main results}\label{main_results}
The main results contained in this work are the quantitative uncertainty principles holding for general Gelfand-Shilov spaces given in Theorem~\ref{general_uncertaintyprinciple}. The first part of this section is devoted to present these new uncertainty principles and to discuss the particular case of spaces of functions with weighted Hermite expansions. In a second part, we deduce from these new uncertainty principles some null-controllability results for abstract evolution equations with adjoint systems enjoying Gelfand-Shilov smoothing effects. Before stating these results, miscellaneous facts and notations need to be presented.
A sequence $\mathcal{M}=(M_p)_{p \in \nn}$ of positive real numbers is said to be \textit{logarithmically convex} if
\begin{equation*}\label{log_conv}
\forall p \geq 1, \quad M_p^2 \leq M_{p+1} M_{p-1},
\end{equation*}
where $\nn$ denotes the set of non-negative integers.
Let $U$ be an open subset of $\rr^d$, with $d \geq 1$. We consider the following class of smooth functions defined on $U$ associated to the sequence $\mathcal{M}$,
\begin{equation*}\label{function_class}
\mathcal{C}_{\mathcal{M}}(U)= \left\{ f \in \mathcal{C}^{\infty}(U, \cc): \quad \forall \beta \in \nn^d, \;  \|\partial_x^{\beta} f \|_{L^{\infty}(U)} \leq M_{|\beta|} \right\}.
\end{equation*}
A logarithmically convex sequence $\mathcal{M}$ is said to be quasi-analytic if the class of smooth functions $\mathcal{C}_{\mathcal{M}}((0,1))$ associated to $\mathcal{M}$ is quasi-analytic, that is, when the only function in $\mathcal{C}_{\mathcal{M}}((0,1))$ vanishing to infinite order at a point in $(0,1)$ is the zero function. A necessary and sufficient condition on the logarithmically convex sequence $\mathcal{M}$ to generate a quasi-analytic class is given by  the Denjoy-Carleman theorem (see e.g. \cite{Koosis}):

\medskip
\begin{theorem}[Denjoy-Carleman] \label{Den_Carl_thm}
Let $\mathcal{M}=(M_p)_{p \in \nn}$ be a logarithmically convex sequence of positive real numbers. The sequence $\mathcal{M}$ defines a quasi-analytic sequence if and only if 
\begin{equation*}
\sum_{p= 1}^{+\infty} \frac{M_{p-1}}{M_p} = + \infty.
\end{equation*}
\end{theorem}
\medskip

Let us now introduce the notion of Bang degree defined in \cite{Bang} and \cite{NSV}, and used by Jaye and Mitkovski in \cite{JayeMitkovski},
\begin{equation}\label{Bang}
\forall 0<t \leq 1, \forall r>0, \quad 0 \leq n_{t, \mathcal{M},r}= \sup\Big\{N \in \nn: \, \sum_{-\log t < n \leq N} \frac{M_{n-1}}{M_n} < r \Big\}\leq +\infty,
\end{equation}
where the sum is taken equal to $0$ when $N=0$. Notice that if $\mathcal{M}$ is quasi-analytic, then the Bang degree $n_{t, \mathcal{M},r}$ is finite for any $0<t \leq 1$ and $r>0$.
This Bang degree allows the authors of \cite{JayeMitkovski} to obtain uniform estimates for $L^2$-functions with fast decaying Fourier transforms and to establish uncertainty principles for a general class of Gevrey spaces. These authors also define
\begin{equation}\label{def_gamma}
\forall p \geq 1, \quad \gamma_{\mathcal{M}}(p) = \sup \limits_{1 \leq j \leq p} j \Big(\frac{M_{j+1} M_{j-1}}{M_j^2} -1\Big) \quad  \text{and} \quad  \Gamma_{\mathcal{M}} (p)= 4 e^{4+4\gamma_{\mathcal{M}}(p)}.
\end{equation}
We refer the reader to the Section~\ref{qa_section} for some examples and useful results about quasi-analytic sequences.
\newpage

\subsection{Some uncertainty principles}\label{uncertainty_principle_general}
\subsubsection{Uncertainty principles in general Gelfand-Shilov spaces}\label{general_GS}
In this section, we study uncertainty principles holding in general Gelfand-Shilov spaces. We consider the following subspaces of smooth functions 
\begin{equation*}\label{gelfandshilov}
GS_{\mathcal{N},\rho} := \Big\{ f \in \mathcal{C}^{\infty}(\rr^d), \quad \sup_{k \in \nn,\ \beta \in \nn^d} \frac{\| \rho(x)^k \partial_x^{\beta} f \|_{L^2(\rr^d)}}{N_{k,|\beta|}} < +\infty \Big\},
\end{equation*}
where $\rho : \rr^d \longrightarrow (0,+\infty)$ is a positive measurable function and $\mathcal{N}=(N_{p,q})_{(p,q) \in \nn^2}$ is a sequence of positive real numbers. Associated to these spaces, are the following semi-norms
\begin{equation*}
\forall f \in GS_{\mathcal{N},\rho}, \quad \|f\|_{GS_{\mathcal{N},\rho}} = \sup_{k \in \nn, \ \beta \in \nn^d} \frac{\| \rho(x)^k \partial_x^{\beta} f \|_{L^2(\rr^d)}}{N_{k,|\beta|}}.
\end{equation*}
When $$\forall x \in \rr^d, \quad \rho(x)= \langle x \rangle= (1+\|x\|^2)^{\frac{1}{2}}$$ and $\mathcal{N}=\big( C^{p+q}(p!)^{\nu} (q!)^{\mu} \big)_{(p,q) \in \nn^2}$ for some $C \geq 1$ and $\mu, \nu >0$ with $\mu+\nu \geq1$, $GS_{\mathcal{N}, \rho}$ is a subspace of the classical Gelfand-Shilov space $\mathcal{S}^{\mu}_{\nu}$, whereas when $\rho \equiv 1$, the space $GS_{\mathcal{N}, \rho}$ characterizes some Gevrey type regularity. We choose here to not discuss this particular case since it is studied in the recent works \cite{AlphonseMartin, JayeMitkovski}. In the following, a positive function $\rho : \rr^d \longrightarrow (0,+\infty)$ is said to be a contraction mapping when there exists $0\leq L <1$ such that
$$\forall x,y \in \rr^d, \quad |\rho(x)-\rho(y)| \leq L \|x-y\|,$$ 
where $\| \cdot \|$ denotes the Euclidean norm. A double sequence of real numbers $\mathcal{N}=(N_{p,q})_{(p,q) \in \nn^2}$ is said to be non-decreasing with respect to the two indexes when
$$\forall p \leq p', \forall q \leq q', \quad N_{p,q} \leq N_{p',q'}.$$
The following result provides some uncertainty principles holding for the spaces $GS_{\mathcal{N}, \rho}$:

\medskip

\begin{theorem}\label{general_uncertaintyprinciple}
Let $0<\gamma \leq 1$, $\mathcal{N}=(N_{p, q})_{(p,q) \in \nn^2} \in (0,+\infty)^{\nn^2}$ be a non-decreasing sequence with respect to the two indexes such that the diagonal sequence $\mathcal{M}=(N_{p,p})_{p \in \nn} \in (0,+\infty)^{\nn}$ defines a logarithmically-convex quasi-analytic sequence and  $\rho : \rr^d \longrightarrow (0,+\infty)$ a positive contraction mapping
such that there exist some constants $m>0$, $R>0$ so that
\begin{equation*}
\forall x \in \rr^d, \quad 0<m \leq \rho(x) \leq R \langle x \rangle.
\end{equation*} 
 Let $\omega$ be a measurable subset of $\rr^d$. If $\omega$ is $\gamma$-thick with respect to $\rho$, then there exist some positive constants $ K=K(d,\rho) \geq 1$, $K'=K'(d,\rho, \gamma)\geq 1$, $r=r(d, \rho) \geq 1$ depending on the dimension $d \geq 1$, on $\gamma$ for the second and on the density $\rho$ such that for all $0<\eps \leq N^2_{0,0}$,
 \begin{equation*}
\forall f \in GS_{\mathcal{N},\rho}, \quad  \|f\|^2_{L^2(\rr^d)} \leq C_{\eps} \|f\|^2_{L^2(\omega)} + \eps \|f\|^2_{GS_{\mathcal{N},\rho}},
 \end{equation*} 
 where
 \begin{equation*}
 C_{\eps}= K' \bigg(\frac{2d}{\gamma} \Gamma_{\mathcal{M}}(2n_{t_0, \mathcal{M}, r}) \bigg)^{4n_{t_0, \mathcal{M}, r} }
 \end{equation*}
 with $n_{t_0, \mathcal{M}, r}$ being defined in \eqref{Bang} and
 \begin{equation*}
 t_0=\frac{\eps^{\frac{1}{2}}}{K N_{d,d}}.
 \end{equation*}
\end{theorem}

\medskip 
It is particularly interesting to notice that Theorem~\ref{general_uncertaintyprinciple} provides a quantitative estimate of the constant $C_{\eps}$ with respect to the different parameters. In specific cases, the Bang degree is easily computable (see Lemma~\ref{ex_qa_sequence}) and an explicit upper bound on the constant $C_{\eps}$ can be obtained. The above uncertainty principles apply in particular to the case of the classical Gelfand-Shilov spaces $S_{\nu}^{\mu}(\rr^d)$ as follows:

\medskip

\begin{theorem}\label{specific_GS_uncertaintyprinciple}
Let $A \geq 1$, $0<\mu \leq 1$, $\nu >0$ with $\mu+\nu \geq 1$ and $0\leq \delta \leq \frac{1-\mu}{\nu}\leq 1$. Let $\rho : \rr^d \longrightarrow (0,+\infty)$ be a positive contraction mapping such that there exist some constants $m>0$, $R>0$ so that
\begin{equation*}
\forall x \in \rr^d, \quad 0<m \leq \rho(x) \leq R{\left\langle x\right\rangle}^{\delta}.
\end{equation*} 
 Let $\omega$ be a measurable subset of $\rr^d$. If $\omega$ is thick with respect to $\rho$, then for all $0<\eps \leq 1$, there exists a positive constant $C_{\eps,A}>0$ such that for all $f \in \mathscr{S}(\rr^d)$,
 \begin{equation}\label{up_schwartz}
 \| f \|^2_{L^2(\rr^d)} \leq C_{\eps,A} \|f\|^2_{L^2(\omega)} + \eps \sup_{p \in \nn, \beta \in \nn^d} \bigg(\frac{\|\langle x\rangle^{p} \partial^{\beta}_{x} f\|_{L^2(\rr^d)}}{A^{p+|\beta|} (p!)^{\nu}(|\beta|!)^{\mu}}\bigg)^2,
 \end{equation} 
where, when $\delta < \frac{1-\mu}{\nu}$, there exists a positive constant $K=K(d, \gamma, \rho,\mu, \nu) \geq 1$ depending on the dimension $d$, $\rho$ and $\nu$ such that
  $$0<C_{\eps,A} \leq e^{K(1-\log \eps +A^{\frac{2}{1-\mu-\delta \nu}})}, $$
whereas, when $\delta= \frac{1-\mu}{\nu}$, there exists a positive constant $K=K(d, \gamma, \rho, \mu, \nu) \geq 1$ depending on the dimension $d$, $\rho$ and $\nu$ such that
 $$0<C_{\eps,A} \leq e^{K(1-\log \eps+\log A)e^{KA^2}}.$$
\end{theorem}
\medskip
Let us notice that the estimate \eqref{up_schwartz} is only relevant when 
\begin{equation*}
\sup_{p \in \nn, \beta \in \nn^d} \frac{\|\langle x\rangle^{p} \partial^{\beta}_x f\|_{L^2(\rr^d)}}{A^{p+|\beta|} (p!)^{\nu}(|\beta|!)^{\mu}} <+\infty,
\end{equation*}
that is, when $f \in GS_{\mathcal{N}, \tilde{\rho}}$, with $\mathcal{N}=(A^{p+q}(p!)^{\nu} (q!)^{\mu})_{(p,q) \in \nn^2}$ and $\tilde{\rho}= \langle \cdot \rangle$. The quantitative estimates given in Theorem~\ref{specific_GS_uncertaintyprinciple} are playing a key role in order to establish the following null-controllability results.

The proof of Theorem~\ref{general_uncertaintyprinciple} is given in Section~\ref{proof_mainprop}. It follows the strategy developed by Kovrijkine in \cite{Kovrijkine}, and its generalization given in Theorem~\ref{Spectral} together with a quantitative result on quasi-analytic functions which is a multidimensional version of \cite[Theorem~B]{NSV} from Nazarov, Sodin and Volberg. Regarding Theorem~\ref{specific_GS_uncertaintyprinciple}, its proof is given in Section~\ref{proof2}. It is a direct application of Theorem~\ref{general_uncertaintyprinciple} together with Lemma~\ref{ex_qa_sequence}.
Next section shows that Theorem~\ref{general_uncertaintyprinciple} also applies to more general sequences.

\subsubsection{Uncertainty principles in symmetric weighted Gelfand-Shilov spaces}\label{up_symmetric_GS}
Let $$\Theta : [0,+\infty) \longrightarrow [0,+\infty),$$ be a non-negative continuous function. We consider the following symmetric weighted Gelfand-Shilov spaces 
\begin{equation*}
GS_{\Theta}= \Big\{ f \in L^2(\rr^d): \quad \|f\|_{GS_{\Theta}} := \Big\|\big(e^{\Theta(|\alpha|)} \langle f, \Phi_{\alpha} \rangle_{L^2(\rr^d)}\big)_{\alpha \in \nn^d}\Big\|_{l^2(\nn^d)} < +\infty \Big\},
\end{equation*}
where $(\Phi_{\alpha})_{\alpha \in \nn^d}$ denotes the Hermite basis of $L^2(\rr^d)$. The definition and basic facts about Hermite functions are recalled in Section~\ref{Hermite_functions}.

Before explaining how the spaces $GS_{\Theta}$ relate to Gelfand-Shilov spaces defined in Section~\ref{gelfand}, the assumptions on the weight function $\Theta$ need to be specify further.
Let us consider the following logarithmically-convex sequence
\begin{equation}\label{lc_sequence}
\forall p \in \nn, \quad M_p= \sup_{t \geq 0} t^pe^{-\Theta(t)}.
\end{equation}
Let $s >0$. We assume that the sequence  $(M_p)_{p \in \nn}$ satisfies the following conditions:
\medskip

\text{(H1)} $\forall p \in \nn, \quad 0<M_p < +\infty$,
\medskip

\text{(H2)} There exist some positive constants $C_{\Theta}>0$, $L_{\Theta}\geq 1$ such that 
\begin{equation*}\label{H2}
\forall p \in \nn, \quad p^{p} \leq C_{\Theta} L_{\Theta}^p M_p,
\end{equation*}
with the convention $0^0=1$, 
\medskip

$\text{(H3)}_s$ The sequence $(M^s_p)_{p \in \nn}$ is quasi-analytic, that is, 
\begin{equation*}
\sum_{p=1}^{+\infty} \Big(\frac{M_{p-1}}{M_p}\Big)^s = +\infty,
\end{equation*}
according to Denjoy-Carleman Theorem.
Under these assumptions, the following Bernstein type estimates hold for the spaces $GS_{\Theta}$:
\medskip

\begin{proposition}\label{bernstein_estim1}
Let $\Theta : [0,+\infty) \longrightarrow [0,+\infty)$ be a non-negative continuous function. If the associated sequence $(M_p)_{p \in \nn}$ in \eqref{lc_sequence} satisfies the assumptions $(H1)$ and $(H2)$, then the space $GS_{\Theta}$ is included in the Schwartz space $\mathscr{S}(\rr^d)$, and for all $0< s \leq 1$, there exists a positive constant $D_{\Theta, d,s}\geq 1$ such that
\begin{multline*}
\forall f \in GS_{\Theta}, \forall r \in [0,+\infty), \forall \beta \in \nn^d,\\
\|\langle x \rangle^{r} \partial_x^{\beta} f \|_{L^2(\rr^d)} \leq (D_{\Theta,d,s})^{1+r+|\beta|} \Big(M_{\left\lfloor \frac{r+1 +|\beta|+(2-s)(d+1)}{2s} \right\rfloor +1}\Big)^s \|f\|_{GS_{\Theta}},
\end{multline*}
where $\lfloor \cdot \rfloor$ denotes the floor function.
\end{proposition}

\medskip
\begin{remark} Let us notice that Proposition~\ref{bernstein_estim1} implies in particular the inclusion of spaces $GS_{\Theta} \subset GS_{\mathcal{N}, \rho_s}$, when $\frac{1}{2} \leq s \leq 1$, with 
$$\forall x \in \rr^d, \quad \rho_s(x)= \langle x \rangle^{2s-1}$$
and 
$$\mathcal{N}= \Big((D_{\Theta,d, s}^{(2s-1)p+q+1} M_{\left\lfloor \frac{(2s-1)p+1 +q+(2-s)(d+1)}{2s} \right\rfloor +1}^s \Big)_{(p,q) \in \nn^2},$$
and the following estimates
\begin{equation*}
\forall f \in GS_{\Theta}, \quad \|f\|_{GS_{\mathcal{N},\rho_s}} \leq \|f \|_{GS_{\Theta}}.
\end{equation*}
\end{remark}
The proof of Proposition~\ref{bernstein_estim1} is given in Appendix (Section~\ref{appendix}).
In order to derive uncertainty principles for functions with weighted Hermite expansions, the sequence $\mathcal{M}$ has in addition to satisfy the assumption $(H3)_s$ for some $\frac{1}{2} \leq s \leq 1$. 

%

The quantitative estimates in Proposition~\ref{bernstein_estim1} together with the uncertainty principles given by Theorem~\ref{general_uncertaintyprinciple} allow us to establish the following estimates:

\medskip

\begin{theorem}\label{uncertainty_principle}
Let $0 \leq \delta \leq 1$ and $\Theta : [0,+\infty) \longrightarrow [0,+\infty)$ be a non-negative continuous function. Let us assume that the associated sequence $(M_p)_{p \in \nn}$ in \eqref{lc_sequence} satisfies the assumptions $\text{(H1)}$, $\text{(H2)}$ and $\text{(H3)}_{\frac{1+\delta}{2}}$. Let $\rho : \rr^d \longrightarrow (0,+\infty)$ be a positive contraction mapping satisfying 
\begin{equation*}
\exists m>0, \exists R>0, \forall x \in \rr^d, \quad 0<m \leq \rho(x) \leq R \left\langle x \right\rangle^{\delta}.
\end{equation*}
If $\omega$ is a measurable subset of $\rr^d$ thick with respect to $\rho$, then there exists a positive constant $\eps_0=\eps_0(\Theta, d, \delta)>0$ such that for all $0<\eps \leq \eps_0$, there exists a positive constant $D_{\eps}=D(d, \Theta, \eps, \delta, \rho)>0$ so that
\begin{equation}\label{uncertainty_principle_sym}
\forall f \in GS_{\Theta}, \quad \|f\|^2_{L^2(\rr^d)} \leq D_{\eps} \|f\|^2_{L^2(\omega)}+ \eps \|f\|_{GS_{\Theta}}^2 .
\end{equation}
\end{theorem}


\medskip

The above result provides some uncertainty principles for functions with weighted Hermite expansions. Its proof is given in Section~\ref{thm_hermite_proof}. Let us point out that it is possible to obtain quantitative estimates on the constant $D_{\eps}$ thanks to the ones in Theorem~\ref{general_uncertaintyprinciple}, and to recover the spectral inequalities for finite combinations of Hermite functions established in \cite{MP} with a constant growing at the same rate with respect to $N$. Indeed, by taking $\Theta(t)=t$ on $[0,+\infty)$, we readily compute that
\begin{equation*}
\forall p \in \nn, \quad M_p= \sup_{t \geq 0} t^p e^{-t} = \Big( \frac{p}{e} \Big)^p
\end{equation*}
and the Stirling's formula provides
\begin{equation*}
M_p \underset{p \to +\infty}{\sim} \frac{p!}{\sqrt{2\pi p}}.
\end{equation*}
It follows that the assumptions $\text{(H1)}$, $\text{(H2)}$ and $\text{(H3)}_1$ are satisfied.
By noticing that
\begin{equation*}
\|f\|^2_{GS_{\Theta}} = \sum_{|\alpha| \leq N} e^{2|\alpha|} |\langle f, \Phi_{\alpha} \rangle_{L^2(\rr^d)} |^2 \leq e^{2N} \|f\|^2_{L^2(\rr^d)},
\end{equation*}
when $N \in \nn$ and $f=\sum_{|\alpha| \leq N} \langle f, \Phi_{\alpha} \rangle \Phi_{\alpha}$, we deduce from \eqref{uncertainty_principle_sym} while taking $\eps = \frac{1}{2} e^{-2N}$ that 
\begin{equation*}
\forall N \in \nn, \forall f \in \mathcal{E}_N, \quad \|f\|^2_{L^2(\rr^d)} \leq 2 D_{\frac{1}{2} e^{-2N}} \|f\|^2_{L^2(\omega)},
\end{equation*}
where $\mathcal{E}_N= \textrm{Span}_{\cc}\big\{\Phi_{\alpha}\big\}_{\alpha \in \nn^d, \, |\alpha|\leq N}$.

We end this section by providing some examples of functions $\Theta$, which define a sequence $\mathcal{M}=(M_p)_{p\in \nn}$ satisfying hypotheses $\text{(H1)}$, $\text{(H2)}$ and $\text{(H3)}_{s}$ for some $\frac{1}{2}\leq s \leq 1$. In \cite[Proposition~4.7]{AlphonseMartin}, Alphonse and the author devise the following examples in the case $s=1$:
\medskip
\begin{proposition}[\cite{AlphonseMartin}, Alphonse \& Martin]\label{ex_theta1}
Let $k\geq1$ be a positive integer and $\Theta_{k,1} : [0,+\infty)\rightarrow[0,+\infty)$ be the non-negative function defined for all $t\geq0$ by 
$$\Theta_{k,1}(t) = \frac t{g(t)(g\circ g)(t)... g^{\circ k}(t)},\quad\text{where}\quad g(t) = \log(e+t),$$
with $g^{\circ k} = g\circ\ldots\circ g$ ($k$ compositions). The associated sequence $\mathcal M^{\Theta_{k,1}}=(M^{\Theta_{k,1}}_p)_{p \in \nn}$ defined in \eqref{lc_sequence} is a quasi-analytic sequence of positive real numbers.
\end{proposition}
\medskip
Let us notice that the assumption $\text{(H2)}$ is satisfied as
\begin{equation*}
\forall k \geq 1, \forall p \in \nn, \quad M^{\Theta_{k,1}}_p= \sup_{t \geq 0} t^p e^{-\Theta_{k,1}(t)} \geq \sup_{t \geq 0} t^p e^{-t} = \Big(\frac{p}{e}\Big)^p.
\end{equation*}
Proposition~\ref{ex_theta1} allows to provide some examples for the cases $\frac{1}{2} \leq s \leq 1$:
\medskip
\begin{proposition}\label{ex_qa_bertrand}
Let $k\geq1$ be a positive integer, $\frac{1}{2} \leq s \leq 1$ and $\Theta_{k,s} : [0,+\infty)\rightarrow[0,+\infty)$ be the non-negative function defined for all $t\geq0$ by 
$$\Theta_{k,s}(t) = \frac {t^s}{g(t)(g\circ g)(t)... g^{\circ k}(t)},\quad\text{where}\quad g(t) = \log(e+t),$$
with $g^{\circ k} = g\circ\ldots\circ g$ ($k$ compositions). The associated sequence $\mathcal M^{\Theta_{k,s}}=(M^{\Theta_{k,s}}_p)_{p \in \nn}$ defined in \eqref{lc_sequence} satisfies the assumptions $\text{(H1)}$, $\text{(H2)}$ and $\text{(H3)}_{s}$.
\end{proposition}
\medskip
The proof of Proposition~\ref{ex_qa_bertrand} is given in Section~\ref{qa_section}.
\subsection{Applications to the null-controllability of evolution equations}\label{null_controllability_results}
This section is devoted to state some null-controllability results for evolution equations whose adjoint systems enjoy Gelfand-Shilov smoothing effects. Before presenting these results, let us recall the definitions and classical facts about controllability.
The notion of null-controllability is defined as follows:

\medskip

\begin{definition} [Null-controllability] Let $P$ be a closed operator on $L^2(\rr^d)$, which is the infinitesimal generator of a strongly continuous semigroup $(e^{-tP})_{t \geq 0}$ on $L^2(\rr^d)$, $T>0$ and $\omega$ be a measurable subset of $\mathbb{R}^d$. 
The evolution equation 
\begin{equation}\label{syst_general}
\left\lbrace \begin{array}{ll}
(\partial_t + P)f(t,x)=u(t,x)\un_{\omega}(x), \quad &  x \in \mathbb{R}^d,\ t>0, \\
f|_{t=0}=f_0 \in L^2(\rr^d),                                       &  
\end{array}\right.
\end{equation}
is said to be {\em null-controllable from the set $\omega$ in time} $T>0$ if, for any initial datum $f_0 \in L^{2}(\mathbb{R}^d)$, there exists a control function $u \in L^2((0,T)\times\mathbb{R}^d)$ supported in $(0,T)\times\omega$, such that the mild \emph{(}or semigroup\emph{)} solution of \eqref{syst_general} satisfies $f(T,\cdot)=0$.
\end{definition}

\medskip

By the Hilbert Uniqueness Method, see \cite{coron_book} (Theorem~2.44) or \cite{JLL_book}, the null-controllability of the evolution equation \eqref{syst_general} is equivalent to the observability of the adjoint system 
\begin{equation} \label{adj_general}
\left\lbrace \begin{array}{ll}
(\partial_t + P^*)g(t,x)=0, \quad & x \in \mathbb{R}^d, \ t>0, \\
g|_{t=0}=g_0 \in L^2(\rr^d),
\end{array}\right.
\end{equation}
where $P^*$ denotes the $L^2(\rr^d)$-adjoint of $P$. 
The notion of observability is defined as follows:

\medskip

\begin{definition} [Observability] Let $T>0$ and $\omega$ be a measurable subset of $\mathbb{R}^d$. 
The evolution equation \eqref{adj_general} is said to be {\em observable from the set $\omega$ in time} $T>0$, if there exists a positive constant $C_T>0$ such that,
for any initial datum $g_0 \in L^{2}(\mathbb{R}^d)$, the mild \emph{(}or semigroup\emph{)} solution of \eqref{adj_general} satisfies
\begin{equation*}\label{eq:observability}
\int\limits_{\mathbb{R}^d} |g(T,x)|^{2} dx  \leq C_T \int\limits_{0}^{T} \Big(\int\limits_{\omega} |g(t,x)|^{2} dx\Big) dt\,.
\end{equation*}
\end{definition}

\medskip
In the following, we shall always derive null-controllability results from observability estimates on adjoint systems.

%

\subsubsection{Null-controllability of evolution equations whose adjoint systems enjoy non symmetric Gelfand-Shilov smoothing effects}\label{null_controllability_non_symmetric_GS}
In this section, we aim at establishing null-controllability results for evolution equations whose adjoint systems enjoy Gelfand-Shilov smoothing effects. We consider $A$ a closed operator on $L^2(\rr^d)$, that is the infinitesimal generator of a strongly continuous contraction semigroup $(e^{-tA})_{t \geq 0}$ on $L^2(\rr^d)$, that is satisfying $$\forall t \geq 0, \forall f \in L^2(\rr^d), \quad \|e^{-tA}f\|_{L^2(\rr^d)} \leq \|f\|_{L^2(\rr^d)},$$ and study the evolution equation
\begin{equation}\label{PDEgeneral}
\left\lbrace \begin{array}{ll}
\partial_tf(t,x) + Af(t,x)=u(t,x)\un_{\omega}(x), \quad &  x \in \mathbb{R}^d, \ t>0, \\
f|_{t=0}=f_0 \in L^2(\rr^d).                                       &  
\end{array}\right.
\end{equation}
We assume that the semigroup $(e^{-tA^*})_{t \geq 0}$ generated by the $L^2(\rr^d)$-adjoint operator $A^*$, enjoys some Gelfand-Shilov smoothing effects for any positive time, that is,
\begin{equation*}\label{GS0}
\forall t >0, \forall f \in L^2(\rr^d), \quad e^{-tA^*}f \in S^{\mu}_{\nu} (\rr^d),
\end{equation*}
for some $\mu, \nu >0$ satisfying $\mu +\nu \geq 1$.
More precisely, we assume that the following quantitative regularizing estimates hold: there exist some constants $C \geq 1$, $r_1>0$, $r_2\geq0$, $0<t_0 \leq 1$ such that 
\begin{multline}\label{GS_estimate}
\forall 0< t \leq t_0, \forall \alpha, \beta \in \nn^d, \forall f \in L^2(\rr^d), \\ 
\|x^{\alpha} \partial^{\beta}_x (e^{-tA^*} f) \|_{L^2(\rr^d)} \leq \frac{C^{1+|\alpha|+|\beta|}}{t^{r_1(|\alpha|+|\beta|)+r_2}} (\alpha!)^{\nu} (\beta !)^{\mu} \|f\|_{L^2(\rr^d)},
\end{multline}
where $\alpha!=\alpha_1!...\alpha_d!$ if $\alpha=(\alpha_1,...,\alpha_d) \in \nn^d$.
In a recent work \cite{Alphonse}, Alphonse studies the smoothing effects of semigroups generated by anisotropic Shubin operators 
\begin{equation*}
\mathcal{H}_{m,k} = (-\Delta_x)^m+|x|^{2k},
\end{equation*}
equipped with domains
\begin{equation*}
D(\mathcal{H}_{m,k})= \left\{f \in L^2(\rr^d) : \mathcal{H}_{m,k} f \in L^2(\rr^d) \right\},
\end{equation*}
when $m,k \geq 1$ are positive integers. This author establishes in \cite[Corollary~2.2]{Alphonse} the following quantitative estimates for fractional anisotropic Shubin operators:
for all $m, k \geq 1$, $s >0$, there exist some positive constants $C\geq 1$, $r_1, r_2>0$, $t_0>0$ such that
\begin{multline*}
\forall 0< t \leq t_0, \forall \alpha, \beta \in \nn^d, \forall f \in L^2(\rr^d),  \\ \|x^{\alpha} \partial_x^{\beta} \big(e^{-t\mathcal{H}_{m,k}^s} f\big) \|_{L^2(\rr^d)} \leq \frac{C^{1+|\alpha|+|\beta|}}{t^{r_1(|\alpha|+|\beta|)+r_2}} (\alpha!)^{\nu_{m,k,s}} (\beta !)^{\mu_{m,k,s}} \|f\|_{L^2(\rr^d)},
\end{multline*}
with
\begin{equation*}
\nu_{m,k,s}= \max\Big(\frac{1}{2sk}, \frac{m}{k+m} \Big) \quad \text{and} \quad \mu_{m,k,s}= \max\Big(\frac{1}{2sm}, \frac{k}{k+m} \Big).
\end{equation*}
Thanks to these quantitative estimates and the general result of null-controllability for evolution equations whose adjoint systems enjoy symmetric Gelfand-Shilov smoothing effects in \cite[Theorem~2.5]{MP}, Alphonse derives in \cite[Theorem~2.3]{Alphonse} a sufficient  growth condition on the density $\rho$ to ensure the null-controllability of evolution equations associated to the Shubin operators $\mathcal{H}_{l,l}$, with $l \geq 1$, in any positive time from any measurable thick control subset with respect to $\rho$. In this work, we extend this result to general Shubin operators $\mathcal{H}_{m,k}$, with $m, k \geq 1$, and more generally establish the null-controllability of evolution equations whose adjoint systems enjoy quantitative smoothing effects in specific Gelfand-Shilov spaces $S_{\nu}^{\mu}$.

The following result shows that null-controllability holds for the evolution equations \eqref{PDEgeneral} when the parameter $\delta$ ruling the growth of the density is strictly less than the critical parameter $\delta^*=\frac{1-\mu}{\nu}$.
\begin{theorem}\label{observability_result}
Let $(A,D(A))$ be a closed operator on $L^2(\rr^d)$ which is the infinitesimal generator of a strongly continuous contraction semigroup $(e^{-tA})_{t \geq 0}$ on $L^2(\rr^d)$ whose $L^2(\rr^d)$-adjoint generates a semigroup satisfying the quantitative smoothing estimates \eqref{GS_estimate} for some $0<\mu <1 $, $\nu >0$ such that $\mu+\nu \geq 1$. Let $\rho : \rr^d \longrightarrow (0,+\infty)$ be a contraction mapping such that there exist some constants $0 \leq \delta < \frac{1-\mu}{\nu}$, $m>0$, $R>0$ so that
\begin{equation*}
\forall x \in \rr^d, \quad 0< m \leq \rho(x) \leq R \langle x \rangle^{\delta}.
\end{equation*}
If $\omega \subset \rr^d$ is a measurable subset thick with respect to the density $\rho$, the evolution equation
\begin{equation*}\label{PDEadjoint}
\left\lbrace \begin{array}{ll}
\partial_tf(t,x) + Af(t,x)=u(t,x)\un_{\omega}(x), \quad &  x \in \mathbb{R}^d, \ t>0, \\
f|_{t=0}=f_0 \in L^2(\rr^d),                                       &  
\end{array}\right.
\end{equation*}
is null-controllable from the control subset $\omega$ in any positive time $T>0$; and equivalently, the adjoint system
\begin{equation*}\label{PDEadjoint}
\left\lbrace \begin{array}{ll}
\partial_t g(t,x) + A^*g(t,x)=0, \quad &  x \in \mathbb{R}^d, \ t>0, \\
g|_{t=0}=g_0 \in L^2(\rr^d),                                       &  
\end{array}\right.
\end{equation*}
is observable from the control subset $\omega$ in any positive time $T>0$. More precisely, there exists a positive constant $K=K(d, \rho, \delta, \mu, \nu) \geq 1$ such that 
\begin{equation*}
\forall g \in L^2(\rr^d), \forall T>0, \quad \|e^{-TA^*}g\|^2_{L^2(\rr^d)} \leq K \exp\Big(\frac{K}{T^{\frac{2r_1}{1-\mu-\delta \nu}}}\Big) \int_0^T \|e^{-tA^*}g\|^2_{L^2(\omega)} dt.
\end{equation*}
\end{theorem}

\medskip
The proof of Theorem~\ref{observability_result} is given in Section~\ref{proof_obs}. It is derived from the uncertainty principles established in Theorem~\ref{specific_GS_uncertaintyprinciple} while revisiting the adapted Lebeau-Robbiano method used in \cite[Section~8.3]{BeauchardPravdaStarov} with some inspiration taken from the work of Miller \cite{Miller}.

Contrary to \cite[Theorem~2.5]{MP}, where the authors take advantage of the characterization of symmetric Gelfand-Shilov spaces through the decomposition into Hermite basis, let us stress that the above proof does not rely on a similar characterization of general Gelfand-Shilov spaces through the decomposition into an Hilbert basis composed by the eigenfunctions of a suitable operator.
In the critical case $\delta = \delta^*$, the null-controllability of the evolution equation \eqref{PDEgeneral} whose adjoint system enjoys quantitative smoothing estimates in the Gelfand-Shilov space $S_{\nu}^{\mu}$ is still an open problem.

As mentionned above, the general Shubin operators $\mathcal{H}_{m,k}$ are self-adjoint and generate strongly continuous semigroups on $L^2(\rr^d)$, which enjoy quantitative smoothing effects. Consequently, Theorem~\ref{observability_result} can be directly applied to obtain the following null-controllability results:

\medskip

\begin{corollary}
Let $m,k \geq 1$ be positive integers, $s> \frac{1}{2m}$ and 
\begin{equation*}
\delta^*_{m,k,s} := \left\lbrace \begin{array}{ll}
1 &  \text{if } s \geq \frac{m+k}{2mk}, \\
\frac{k}{m} (2sm-1)                       &  \text{if } \frac{1}{2m} < s \leq \frac{m+k}{2mk}.
\end{array}\right.
\end{equation*}
Let $\rho : \rr^d \longrightarrow (0,+\infty)$ be a contraction mapping such that there exist some constants $0 \leq \delta < \delta^*$, $m>0$, $R>0$ so that
\begin{equation*}
\forall x \in \rr^d, \quad 0< m \leq \rho(x) \leq R \langle x \rangle^{\delta}.
\end{equation*}
If $\omega \subset \rr^d$ is a measurable subset thick with respect to the density $\rho$, the evolution equation associated to the fractional Shubin operator
\begin{equation*}\label{PDE_shubin}
\left\lbrace \begin{array}{ll}
\partial_tf(t,x) + \mathcal{H}_{m,k}^s f(t,x)=u(t,x)\un_{\omega}(x), \quad &  x \in \mathbb{R}^d, \ t>0, \\
f|_{t=0}=f_0 \in L^2(\rr^d),                                       &  
\end{array}\right.
\end{equation*}
is null-controllable from the control subset $\omega$ in any time $T>0$.
\end{corollary}

\medskip

\section{Proof of the uncertainty principles}
\subsection{Proof of Theorem~\ref{general_uncertaintyprinciple}}\label{proof_mainprop}
This section is devoted to the proof of Theorem~\ref{general_uncertaintyprinciple}.
Let $\rho : \rr^d \longrightarrow (0,+\infty)$ be a positive contraction mapping such that there exist some positive constants $m>0$, $R>0$ so that
\begin{equation}\label{rho_condi}
\forall x \in \rr^d, \quad 0<m \leq \rho(x) \leq R \left\langle x \right\rangle.
\end{equation}
Let $\omega \subset \rr^d$ be a measurable subset $\gamma$-thick with respect to the density $\rho$, that is,
\begin{equation*}\label{thick_rho}
\exists 0 < \gamma \leq 1, \forall x \in \rr^d, \quad |\omega \cap B(x,\rho(x))| \geq \gamma  |B(x,\rho(x))|=\gamma \rho(x)^d |B(0,1)|,
\end{equation*}
where $B(x,r)$ denotes the Euclidean ball centered at $x \in \rr^d$ with radius $r>0$, and where $|\cdot|$ denotes the Lebesgue measure.
Since $\rho$ is a positive contraction mapping, Lemma~\ref{slowmet} in Appendix and the remark made after the statement of this result show that the family of norms $(\|\cdot\|_x)_{x \in \rr^d}$ given by
\begin{equation*}
\forall x \in \rr^d, \forall y \in \rr^d, \quad \|y\|_x=\frac{\|y\|}{\rho(x)},
\end{equation*}
where $\|\cdot\|$ denotes the Euclidean norm in $\rr^d$, defines a slowly varying metric on $\rr^d$.

\subsection{Step 1. Bad and good balls}  By using Theorem~\ref{slowmetric} in appendix, we can find a sequence $(x_k)_{k \geq 0}$ in $\rr^d$ such that 
\begin{multline}\label{recov}
\exists K_0 \in \nn, \forall (i_1, ..., i_{K_0+1}) \in \nn^{K_0+1} \textrm{ with } i_k \neq i_l \textrm{ if }1 \leq k \neq l \leq K_0+1, \\ \bigcap \limits_{k=1}^{K_0+1} {B_{i_k}}=\emptyset
\end{multline}
and
\begin{equation}\label{recov1}
 \rr^d=\bigcup_{k=0}^{+\infty} {B_k},
\end{equation}
where 
\begin{equation}\label{asdf1}
B_k=\{y \in \rr^d:\ \|y-x_k\|_{x_k} <1\}=\{y \in \rr^d:\ \|y-x_k\| <\rho(x_k)\}=B(x_k,\rho(x_k)). 
\end{equation}
Let us notice from Theorem~\ref{slowmetric} that the non-negative integer $K_0 =K_0(d, L)$ only depends on the dimension $d$ and $L$ the Lipschitz constant of $\rho$, since the constant $C \geq 1$ appearing in slowness condition (\ref{equiv}) can be taken equal to $C=\frac{1}{1-L}$.  
It follows from (\ref{recov}) and (\ref{recov1}) that 
\begin{equation}\label{asdf2}
\forall x \in \rr^d, \quad 1 \leq \sum \limits_{k=0}^{+\infty} \mathbbm{1}_{B_k} (x) \leq K_0,
\end{equation}
where $\mathbbm{1}_{B_k}$ denotes the characteristic function of $B_k$.
We deduce from (\ref{asdf2}) and the Fubini-Tonelli theorem that for all $g \in L^2(\rr^d)$,
\begin{equation*}
\|g\|_{L^2(\rr^d)}^2 = \int_{\rr^d}|g(x)|^2dx \leq \sum_{k=0}^{+\infty}\int_{B_k}|g(x)|^2dx \leq K_0 \|g\|_{L^2(\rr^d)}^2.
\end{equation*}
Let $f \in GS_{\mathcal{N}, \rho} \setminus \{0\}$ and $\eps>0$. We divide the family of balls $(B_k)_{k \geq 0}$ into families of good and bad balls. A ball $B_k$, with $k \in \nn$, is said to be good if it satisfies 
\begin{equation}\label{good_ball}
\forall p \in \nn, \forall \beta \in \nn^d, \quad 
\int_{B_k} |\rho(x)^{p} \partial_{x}^{\beta} f(x)|^2 dx  
\leq \eps^{-1}2^{2(p+|\beta|)+d+1} K_0 N^{2}_{p, |\beta|} \int_{B_k} |f(x)|^2 dx,
\end{equation}
On the other hand, a ball $B_k$, with $k \in \nn$, which is not good, is said to be bad, that is, when 
\begin{multline}\label{bad_ball}
\exists (p_0, \beta_0) \in \nn \times \nn^d, \\
\int_{B_k} |\rho(x)^{p_0} \partial_{x}^{\beta_0} f(x)|^2 dx 
> \eps^{-1}2^{2(p_0+|\beta_0|)+d+1} K_0 N^{2}_{p_0, |\beta_0|}  \int_{B_k} |f(x)|^2 dx.
\end{multline}
If $B_k$ is a bad ball, it follows from \eqref{bad_ball} that there exists $(p_0, \beta_0) \in \nn \times \nn^d$ such that
\begin{multline}\label{gh05}
 \int_{B_k}|f(x)|^2dx \leq
\frac{\eps}{2^{2(p_0+|\beta_0|)+d+1} K_0 N^{2}_{p_0, |\beta_0|}}\int_{B_k} \rho(x)^{2p_0}|\partial_x^{\beta_0}f(x)|^2dx  \\ \leq  \sum_{(p,\beta) \in \nn\times \nn^d} \frac{\eps}{2^{2(p+|\beta|)+d+1} K_0 N^{2}_{p, |\beta|} }\int_{B_k}\rho(x)^{2p}|\partial_x^{\beta}f(x)|^2dx.
\end{multline}
By summing over all the bad balls and by using from (\ref{recov}) that
\begin{equation*}
\mathbbm{1}_{\bigcup_{\textrm{bad balls}} B_k} \leq \sum_{\textrm{bad balls}} \mathbbm{1}_{B_k} \leq K_0 \mathbbm{1}_{\bigcup_{\textrm{bad balls}} B_k},
\end{equation*} we deduce from (\ref{gh05}) and the Fubini-Tonelli theorem that 
\begin{multline}\label{gh6}
\int_{\bigcup_{\textrm{bad balls}} B_k}|f(x)|^2dx \leq \sum_{\textrm{bad balls}}\int_{B_k}|f(x)|^2dx 
\\ \leq \sum_{(p,\beta) \in \nn\times \nn^d}\frac{\eps}{2^{2(p+|\beta|)+d+1} N^{2}_{p, |\beta|} } \int_{\bigcup_{\textrm{bad balls}}  B_k} \hspace{-8mm}  |\rho(x)^{p} \partial_x^{\beta} f(x)|^2dx.
\end{multline}
By using that the number of solutions to the equation $p+\beta_1+...+\beta_{d}=m$, with $m \geq 0$, $d \geq 1$ and unknowns $p \in \nn$ and $\beta=(\beta_1,...,\beta_d) \in \nn^{d}$, is given by $\binom{m+d}{m}$, we obtain from (\ref{gh6}) that 
\begin{multline}\label{gh6y}
\int_{\bigcup_{\textrm{bad balls}} B_k}|f(x)|^2dx \leq \eps \sum_{m \geq 0} \frac{\binom{m+d}{m}}{2^{d+1} 4^m} \|f\|_{GS_{\mathcal{N},\rho}}^2 \\ 
\leq \eps \sum_{m \geq 0} \frac{2^{m+d}}{2^{d+1} 4^m} \|f\|^2_{GS_{\mathcal{N},\rho}}= \eps \|f\|^2_{GS_{\mathcal{N},\rho}},
\end{multline}
since 
\begin{equation*}\label{gh45}
\binom{m+d}{m} \leq \sum_{j=0}^{m+d}\binom{m+d}{j}=2^{m+d}.
\end{equation*}
Recalling from (\ref{recov1}) that 
$$ 1 \leq \mathbbm{1}_{\bigcup_{\textrm{bad balls}}B_k}+ \mathbbm{1}_{\bigcup_{\textrm{good balls}}B_k},$$ 
we notice that  
\begin{equation}\label{asdf5}
\|f\|_{L^2(\rr^d)}^2 \leq \int_{\bigcup_{\textrm{good balls}} B_k}|f(x)|^2dx+ \int_{\bigcup_{\textrm{bad balls}} B_k}|f(x)|^2dx.
\end{equation}
It follows from (\ref{gh6y}) and (\ref{asdf5}) that 
\begin{equation}\label{gh7}
\|f\|_{L^2(\rr^d)}^2 \leq \int_{\bigcup_{\textrm{good balls}} B_k}|f(x)|^2dx+ \eps \|f\|^2_{GS_{\mathcal{N},\rho}}.
\end{equation}

\subsection{Step 2. Properties on good balls}
As the ball $B(0,1)$ is an Euclidean ball, the Sobolev embedding 
$$W^{d,2}(B(0,1)) \xhookrightarrow{} L^{\infty}(B(0,1)),$$
see e.g.~\cite{adams} (Theorem~4.12), implies that there exists a positive constant $C_{d}\geq 1$ depending only the dimension $d \geq 1$ such that 
\begin{equation}\label{sobolev}
\forall u \in W^{d,2}(B(0,1)), \quad 
\|u\|_{L^{\infty}(B(0,1))} \leq C_{d} \|u\|_{W^{d,2}(B(0,1))}.
\end{equation}
By translation invariance and homogeneity of the Lebesgue measure, it follows from (\ref{rho_condi}), (\ref{asdf1}) and (\ref{sobolev}) that for all $u \in {W^{d,2}(B_k)}$,
\begin{multline*}
\|u\|^2_{L^{\infty}(B_k)}=\|x \mapsto u(x_k+x \rho(x_k))\|^2_{L^{\infty}(B(0,1))}
 \leq C_{d}^2 \|x \mapsto u(x_k+x \rho(x_k))\|^2_{W^{d,2}(B(0,1))} \\
 =C_{d}^2 \sum_{\substack{\alpha \in \nn^d, \\ |\alpha| \leq d}} \int_{B_k} \rho(x_k)^{2|\alpha|-d} |\partial^{\alpha}_x u(x)|^2 dx
 =C_{d}^2 \sum_{\substack{\alpha \in \nn^d, \\ |\alpha| \leq d}} \int_{B_k}m^{2|\alpha|-d} \Big( \frac{\rho(x_k)}{m} \Big)^{2|\alpha|-d} |\partial^{\alpha}_x u(x)|^2 dx\end{multline*}
and
\begin{multline}{\label{se1}}
\|u\|^2_{L^{\infty}(B_k)}
 \leq C_{d}^2 \max(m,m^{-1})^d \sum_{\substack{\alpha \in \nn^d, \\ |\alpha| \leq d}} \int_{B_k} \Big( \frac{\rho(x_k)}{m} \Big)^{d} |\partial^{\alpha}_x u(x)|^2 dx\\
 = C_{d}^2 \max(1,m^{-1})^{2d} \rho(x_k)^{d} \sum_{\substack{\alpha \in \nn^d, \\ |\alpha| \leq d}} \int_{B_k} |\partial^{\alpha}_x u(x)|^2 dx.
\end{multline}
We deduce from (\ref{se1}) that for all $u \in {W^{d,2}(B_k)}$,
\begin{equation}{\label{se2}}
\|u\|_{L^{\infty}(B_k)} \leq C_{d} \max(1,m^{-1})^{d} \rho(x_k)^{\frac{d}{2}} \|u\|_{W^{d,2}(B_k)}.
\end{equation}
Let $B_k$ be a good ball. By using the fact that $\rho$ is a $L$-Lipschitz function, we notice that
\begin{equation}{\label{equi}}
\forall x \in B_k=B(x_k,\rho(x_k)), \quad 0 < \rho(x_k) \leq \frac{1}{1-L} \rho(x).
\end{equation} 
We deduce from (\ref{se2}) and (\ref{equi}) that for all $\beta \in \nn^d$ and $k \in \nn$ such that $B_k$ is a good ball
\begin{align}\label{gh30}
& \  \rho(x_k)^{|\beta|+ \frac{d}{2}}\|\partial_x^{\beta}f\|_{L^{\infty}(B_k)} \\ \notag
\leq & \ C_d \max(1,m^{-1})^{d} \rho(x_k)^{|\beta|+ d}\Big(\sum_{\substack{\tilde{\beta} \in \nn^d, \ |\tilde{\beta}| \leq d}}\|\partial_x^{\beta+\tilde{\beta}}f\|^2_{L^{2}(B_k)}\Big)^{\frac{1}{2}}\\ \notag
=  & \ C_d \max(1,m^{-1})^{d}  \Big(\sum_{\substack{\tilde{\beta} \in \nn^d, \ |\tilde{\beta}| \leq d}}\| \rho(x_k)^{|\beta|+ d}\partial_x^{\beta+\tilde{\beta}}f\|^2_{L^{2}(B_k)}\Big)^{\frac{1}{2}} \\ \notag
\leq & \ C_d \max(1,m^{-1})^{d} \frac{1}{(1-L)^{|\beta|+d}} \Big(\sum_{\substack{\tilde{\beta} \in \nn^d, \ |\tilde{\beta}| \leq d}}\| \rho(x)^{|\beta|+ d}\partial_x^{\beta+\tilde{\beta}}f\|^2_{L^{2}(B_k)}\Big)^{\frac{1}{2}}.
\end{align}
By using (\ref{rho_condi}) and the definition of good balls (\ref{good_ball}), it follows from (\ref{gh30}) and the fact that $\mathcal{N}$ is non-decreasing with respect to the two indexes that for all $\beta \in \nn^d$ and $k \in \nn$ such that $B_k$ is a good ball
\begin{align}\label{asdf7}
 & \ \rho(x_k)^{|\beta|+ \frac{d}{2}}\|\partial_x^{\beta}f\|_{L^{\infty}(B_k)} \\ \notag
\leq & \ C_d \max(1,m^{-1})^{d} \frac{1}{(1-L)^{|\beta|+d}} \Big(\sum_{\substack{\tilde{\beta} \in \nn^d, \\ |\tilde{\beta}| \leq d}} \eps^{-1} 2^{2(2|\beta|+ |\tilde{\beta}|)+3d+1} K_0 N^{2}_{|\beta|+d, |\beta|+ |\tilde{\beta}|}  \|f\|^2_{L^{2}(B_k)}\Big)^{\frac{1}{2}} \\ \notag
\leq & \ \eps^{-\frac{1}{2}} K_{d,m, L}  \Big(\frac{4}{1-L} \Big)^{|\beta|} N_{|\beta|+d, |\beta|+d} \|f\|_{L^{2}(B_k)},
\end{align}
with $$K_{d,m, L}= C_d \max(1,m^{-1})^d \sqrt{2K_0}\Big(\frac{4\sqrt{2}}{1-L} \Big)^d (d+1)^{\frac{d}{2}}  \geq 1,$$
since $C_d \geq 1$.
\subsection{Step 3 : Recovery of the $L^2(\rr^d)$-norm.} Let $B_k$ be a good ball. Let us assume that $\|f\|_{L^2(B_k)} \neq0$. We can therefore define the following function

\begin{equation}\label{gh13b}
\forall y \in B(0,1), \quad \phi(y)=\eps^{\frac{1}{2}}\rho(x_k)^{\frac{d}{2}}\frac{f(x_k+\rho(x_k) y)}{ K_{d,m,L} N_{d,d} \|f\|_{L^2(B_k)}}.
\end{equation}
We observe that 
\begin{equation*}
\|\phi\|_{L^{\infty}(B(0,1))} = \eps^{\frac{1}{2}}\rho(x_k)^{\frac{d}{2}}\frac{\|f\|_{L^{\infty}(B_k)}}{K_{d,m,L} N_{d,d} \|f\|_{L^2(B_k)}} \geq \frac{\eps^{\frac{1}{2}}}{|B(0,1)|^{\frac{1}{2}}K_{d,m,L} N_{d,d}},
\end{equation*}
and
\begin{equation}\label{qa1}
\forall \beta \in \nn^d, \quad \| \partial_x^{\beta} \phi \|_{L^{\infty}(B(0,1))}  = \frac{\eps^{\frac{1}{2}} \rho(x_k)^{|\beta|+\frac{d}{2}}  \|\partial_x^{\beta} f\|_{L^{\infty}(B_k)}}{K_{d,m,L} N_{d,d}\|f\|_{L^2(B_k)}}.
\end{equation}
It follows from \eqref{asdf7} and \eqref{qa1} that 
\begin{equation}\label{qa2}
\forall \beta \in \nn^d, \quad \| \partial_x^{\beta} \phi \|_{L^{\infty}(B(0,1))} \leq \Big( \frac{4}{1-L} \Big)^{|\beta|}  \frac{N_{|\beta|+d,|\beta|+d}}{N_{d,d}}.
\end{equation}
We deduce from \eqref{qa2} that $\phi \in \mathcal{C}_{\mathcal{M}'}(B(0,1))$ with $$\mathcal{M}'=(M'_p)_{p \in \nn}= \Big(\Big( \frac{4}{1-L} \Big)^{p}  \frac{N_{p+d,p+d}}{N_{d,d}}\Big)_{p \in \nn}.$$ The assumption that the diagonal sequence $\mathcal{M}=(N_{p,p})_{p \in \nn}$ is logarithmically-convex and quasi-analytic implies that these two properties hold true as well for the sequence $\mathcal{M}'$. Indeed, the logarithmic convexity of $\mathcal{M}'$ is straightforward and since
\begin{equation*}
\sum_{p=0}^{+\infty} \frac{M'_p}{M'_{p+1}} = \frac{1-L}{4} \sum_{p=d}^{+\infty} \frac{N_{p,p}}{N_{p+1,p+1}},
\end{equation*}
we deduce from the quasi-analyticity of $\mathcal{M}$ and from the Denjoy-Carleman's Theorem (Theorem~\ref{Den_Carl_thm}) that the sequence $\mathcal{M}'$ is also quasi-analytic.
 Furthermore, we observe from the definition of the Bang degree \eqref{Bang} and the equality
 \begin{equation*}
 \forall 0<t \leq 1, \forall n \in \nn^*, \quad \sum_{-\log t < p \leq n} \frac{M'_{p-1}}{M'_{p}} = \frac{1-L}{4} \sum_{-\log(t e^{-d})< p \leq n+d} \frac{N_{p-1,p-1}}{N_{p,p}}
 \end{equation*}
  that
\begin{equation}\label{bang1309}
\forall 0< t \leq 1, \quad n_{t, \mathcal{M}',2e d} = n_{te^{-d}, \mathcal{M}, \frac{8d}{1-L}e} -d \leq n_{te^{-d}, \mathcal{M}, \frac{8d}{1-L}e}.
\end{equation}
Setting
\begin{equation*}\label{m_10}
E_k=\Big\{\frac{x-x_k}{\rho(x_k)} \in B(0,1): \ x \in B_k \cap \omega \Big\} \subset B(0,1),
\end{equation*}
we notice from \eqref{asdf1} that
\begin{equation}\label{m_11}
|E_k| = \frac{|\omega \cap B_k|}{\rho(x_k)^d} \geq  \frac{\gamma |B_k|}{\rho(x_k)^d}  \geq \gamma |B(0,1)| >0,
\end{equation}
since $\omega$ is $\gamma$-thick with respect to $\rho$ and $B_k=B(x_k,\rho(x_k))$. From now on, we shall assume that 
\begin{equation}\label{small_eps}
0< \eps \leq N^2_{0,0}.
\end{equation}
 We deduce from \eqref{bang1309} and Proposition~\ref{NSV_multid_L2} applied with the function $\phi$ and the measurable subset $E_k$ of the bounded convex open ball $B(0,1)$ that there exists a positive constant $D_{\eps}=D\big(\eps, \mathcal{N}, d, \gamma, L,m\big)>1$ independent on $\phi$ and $k$ such that 
\begin{equation}\label{qa3}
\int_{B(0,1)} |\phi(x)|^2 dx \leq D_{\eps} \int_{E_k} |\phi(x)|^2 dx,
\end{equation}
with
\begin{equation*}
D_{\eps}= \frac{2}{\gamma} \bigg(\frac{2d}{\gamma} \Gamma_{\mathcal{M}'}\big(2n_{t, \mathcal{M}', 2ed}\big) \bigg)^{4n_{t, \mathcal{M}', 2ed}},
\end{equation*}
and
\begin{equation*}
t= \frac{\eps^{\frac{1}{2}}}{\max\big(1,|B(0,1)|^{\frac{1}{2}}\big)K_{d,m,L} N_{d,d}}.
\end{equation*}
Notice that from \eqref{small_eps}, we have $$0< t \leq 1,$$
since $K_{d,m,L} \geq 1$ and $\mathcal{N}$ is non-decreasing with respect to the two indexes. Let us also notice from the definitions in \eqref{def_gamma} that
$$ \forall n \geq 1, \quad 1 \leq \Gamma_{\mathcal{M}'}(n) \leq \Gamma_{\mathcal{M}}(n+d).$$
We deduce from \eqref{bang1309} and the non-decreasing property of $\Gamma_{\mathcal{M}'}$ that
\begin{align}\label{C_eps}
D_{\eps}  & \leq \frac{2}{\gamma} \bigg(\frac{2d}{\gamma} \Gamma_{\mathcal{M}}\big(2n_{t, \mathcal{M}', 2ed}+d\big) \bigg)^{4n_{t, \mathcal{M}', 2ed}} \\\nonumber  
& \leq \frac{2}{\gamma} \bigg(\frac{2d}{\gamma} \Gamma_{\mathcal{M}}\big(2n_{te^{-d}, \mathcal{M}, \frac{8d}{1-L}e}\big) \bigg)^{4n_{te^{-d}, \mathcal{M}, \frac{8d}{1-L}e}}.
\end{align}
Let us denote $$C_{\eps}=\frac{2}{\gamma} \bigg(\frac{2d}{\gamma} \Gamma_{\mathcal{M}}\big(2n_{te^{-d}, \mathcal{M}, \frac{8d}{1-L}e}\big) \bigg)^{4n_{te^{-d}, \mathcal{M}, \frac{8d}{1-L}e}}.$$
We deduce from \eqref{gh13b}, \eqref{m_11}, \eqref{qa3} and \eqref{C_eps} that
\begin{equation}\label{qa4}
\int_{B_k} |f(x)|^2 dx \leq C_{\eps} \int_{\omega \cap B_k} |f(x)|^2 dx.
\end{equation}
Let us notice that the above estimate holds as well when $\| f \|_{L^2(B_k)}=0$.
By using anew from (\ref{recov}) that 
\begin{equation*}
\mathbbm{1}_{\bigcup_{\textrm{good balls}} B_k}  \leq \sum \limits_{\textrm{good balls}} \mathbbm{1}_{B_k} \leq K_0 \mathbbm{1}_{\bigcup_{\textrm{good balls}} B_k},
\end{equation*}
it follows from (\ref{gh7}) and (\ref{qa4}) that 
\begin{align*}\label{gh56y}
\|f\|_{L^2(\rr^d)}^2 & \leq \int_{\bigcup_{\textrm{good balls}} B_k}|f(x)|^2 dx + \eps \|f\|^2_{GS_{\mathcal{N},\rho}} \\ \notag
& \leq  \sum_{\textrm{good balls}}\|f\|_{L^{2}(B_k)}^2 + \eps \|f\|^2_{GS_{\mathcal{N}, \rho}} \\ \notag  
& \leq C_{\eps} \sum_{\textrm{good balls}} \int_{\omega \cap B_k} |f(x)|^2 dx + \eps \|f\|^2_{GS_{\mathcal{N}, \rho}} \\ \notag  
& \leq K_0 C_{\eps} \int_{\omega \cap \big(\bigcup_{\textrm{good balls}} B_k\big)}|f(x)|^2 dx +\eps \|f\|^2_{GS_{\mathcal{N}, \rho}}.
\end{align*}
The last inequality readily implies that
\begin{equation*}
\|f \|^2_{L^2(\rr^d)} \leq K_0 C_{\eps} \| f \|^2_{L^2(\omega)} +\eps \| f \|^2_{GS_{\mathcal{N}, \rho}}.
\end{equation*}
This ends the proof of Theorem~\ref{general_uncertaintyprinciple}.
\subsection{Proof of Theorem~\ref{specific_GS_uncertaintyprinciple}}\label{proof2}

Let $A \geq 1$, $0< \mu \leq 1$, $\nu >0$ with $\mu+\nu \geq 1$ and $0 \leq \delta \leq \frac{1-\mu}{\nu} \leq 1$. Let $f \in \mathscr{S}(\rr^d)$. We first notice that if the quantity 
\begin{equation*}
\sup_{p \in \nn, \beta \in \nn^d} \frac{\|\langle x \rangle^p \partial_x^{\beta} f \|_{L^2(\rr^d)}}{A^{p+|\beta|} (p!)^{\nu} (|\beta|!)^{\mu}} =+\infty,
\end{equation*}
is infinite, then the result of Theorem~\ref{specific_GS_uncertaintyprinciple} clearly holds. We can therefore assume that 
\begin{equation}\label{bernst_estim_gs}
\sup_{p \in \nn, \beta \in \nn^d} \frac{\|\langle x \rangle^p \partial_x^{\beta} f \|_{L^2(\rr^d)}}{A^{p+|\beta|} (p!)^{\nu} (|\beta|!)^{\mu}} <+\infty.
\end{equation}
By assumption, we have
\begin{equation}\label{rho_assum}
\exists m,R >0, \forall x \in \rr^d, \quad 0< m \leq \rho(x) \leq R \langle x \rangle^{\delta}.
\end{equation}
We deduce from \eqref{bernst_estim_gs}, \eqref{rho_assum} and Lemma~\ref{interpolation} that
\begin{align*}
\forall p \in \nn, \forall \beta \in \nn^d, \quad & \|\rho(x)^p \partial_x^{\beta} f\|_{L^2(\rr^d)} \leq R^p \| \langle x \rangle^{\delta p} \partial^{\beta}_x f\|_{L^2(\rr^d)}\\
& \leq R^p (4^{\nu}e^{\nu}A)^{p+|\beta|} (p!)^{\delta \nu} (|\beta|!)^{\mu} \sup_{q \in \nn, \gamma \in \nn^d} \frac{\|\langle x \rangle^q \partial_x^{\gamma} f \|_{L^2(\rr^d)}}{A^{q+|\gamma|} (q!)^{\nu} (|\gamma|!)^{\mu}} \\
& \leq (4^{\nu}e^{\nu}\max(1,R) A)^{p+|\beta|} (p!)^{\delta \nu} (|\beta|!)^{\mu} \sup_{q \in \nn, \gamma \in \nn^d} \frac{\|\langle x \rangle^q \partial_x^{\gamma} f \|_{L^2(\rr^d)}}{A^{q+|\gamma|} (q!)^{\nu} (|\gamma|!)^{\mu}},
\end{align*}
which implies that
\begin{equation*}
\|f\|_{GS_{\mathcal{N}, \rho}} \leq \sup_{p \in \nn, \beta \in \nn^d} \frac{\|\langle x \rangle^p \partial_x^{\beta} f \|_{L^2(\rr^d)}}{A^{p+|\beta|} (p!)^{\nu} (|\beta|!)^{\mu}} < +\infty,
\end{equation*}
with the non-decreasing sequence $$\mathcal{N}=\Big(\big(4^{\nu} e^{\nu}\max(1,R)A\big)^{p+q}(p!)^{\delta \nu} (q!)^{\mu}\Big)_{(p,q) \in \nn^2}.$$ The assumption $0\leq \delta \leq \frac{1-\mu}{\nu}$ implies that the diagonal sequence $$\mathcal{M}=(M_p)_{p \in \nn}=\Big(\big(4^{\nu} e^{\nu}\max(1,R)A\big)^{2p} (p!)^{\delta \nu + \mu}\Big)_{p \in \nn}$$ is a logarithmically convex quasi-analytic sequence thanks to the Denjoy-Carleman's theorem (Theorem~\ref{Den_Carl_thm}) and since $\delta \nu +\mu \leq 1$. Since $\omega$ is assumed to be $\gamma$-thick with respect to $\rho$ for some $0<\gamma \leq 1$, we deduce from Theorem~\ref{general_uncertaintyprinciple} that there exist some constants $K=K(d, \rho, \delta, \mu, \nu)\geq 1, K'=K'(d, \rho, \gamma)\geq1, r=r(d, \rho) \geq 1$ so that for all $0< \eps \leq M^2_0=1$,
\begin{multline}\label{appl_thm_up}
\|f\|^2_{L^2(\rr^d)} \leq C_{\eps} \| f\|^2_{L^2(\omega)} + \eps \|f\|^2_{GS_{\mathcal{N},\rho}} \\
\leq C_{\eps} \|f\|^2_{L^2(\omega)} + \eps \bigg(\sup_{p \in \nn, \beta \in \nn^d} \frac{\|\langle x \rangle^p \partial_x^{\beta} f \|_{L^2(\rr^d)}}{A^{p+|\beta|} (p!)^{\nu} (|\beta|!)^{\mu}}\bigg)^2,
\end{multline}
where
\begin{equation}\label{c_eps1609}
C_{\eps} =K' \bigg(\frac{2d}{\gamma}\Gamma_{\mathcal{M}}(2n_{t_0,\mathcal{M},r}) \bigg)^{4n_{t_0,\mathcal{M}, r}}
\end{equation}
with
 \begin{equation*}
 0< t_0=\frac{\eps^{\frac{1}{2}}}{K A^{2d}} \leq 1.
 \end{equation*}
 Direct computations
 \begin{equation*}
\forall N \geq 1, \quad \sum_{p > -\log t_0}^N \frac{M_{p-1}}{M_{p}} = (4^{\nu} e^{\nu}\max(1,R) A)^{-2} \sum_{p > -\log t_0}^{N} \frac{(p-1)!^{\delta \nu+\mu}}{p!^{\delta \nu+\mu}}, 
 \end{equation*}
 show that
 \begin{equation*}
 n_{t_0, \mathcal{M}, r} = n_{t_0, \mathcal{M}_{\delta \nu+\mu}, r'A^2},
 \end{equation*}
 with 
 \begin{equation*}
 \mathcal{M}_{\delta \mu+\nu} = \big((p!)^{\delta \nu + \mu} \big)_{p \in \nn} \quad \text{and} \quad r'=r 16^{\nu} e^{2\nu} \max(1,R^2).
 \end{equation*}
By using from Lemma~\ref{ex_qa_sequence} that $\Gamma_{\mathcal{M}_{\delta \nu +\mu}}$ is bounded, it follows that there exists a positive constant $D \geq 1$ such that 
 \begin{equation*}
 \forall n \in \nn^*, \quad \Gamma_{\mathcal{M}}(n) = \Gamma_{\mathcal{M}_{\delta \nu +\mu}}(n) \leq D.
 \end{equation*}
 By using anew Lemma~\ref{ex_qa_sequence} and \eqref{c_eps1609}, we deduce that when $0 \leq \delta \nu+\mu <1$,
 \begin{align}\label{ceps}
 0< C_{\eps} & \leq K'\bigg(\frac{2d}{\gamma} D \bigg)^{4n_{t_0, \mathcal{M}_{\delta \nu +\mu}, r' A^2}}\\ \nonumber
  & \leq K'\bigg(\frac{2d}{\gamma} D \bigg)^{ 2^{\frac{1}{1-\delta \nu-\mu}+2} \big( 1-\log t_0+(A^2 r')^{\frac{1}{1-\delta \nu-\mu}}\big)} \\ \nonumber
 & =K'\bigg(\frac{2d}{\gamma} D \bigg)^{2^{\frac{1}{1-\delta \nu-\mu}+2} \big(1+\log K+ 2d\log A-\frac{1}{2}\log \eps+(A^2 r')^{\frac{1}{1-\delta \nu-\mu}} \big)}.
 \end{align}
 Since $0 \leq \log A \leq A \leq A^{\frac{2}{1-\delta \nu-\mu}}$ and $\log \eps \leq 0$, it follows from \eqref{ceps} that
 \begin{equation*}\label{ceps2}
 0< C_{\eps} \leq K'\bigg(\frac{2d}{\gamma} D \bigg)^{D'\big(1-\log \eps +A^{\frac{2}{1-\delta \nu-\mu}} \big)},
 \end{equation*}
 with $D'= 2^{\frac{1}{1-\delta \nu-\mu}+2}\max \Big(1+\log K, 2d+ {r'}^{\frac{1}{1-\delta \nu-\mu}},\frac{1}{2}\Big)$. 
 On the other hand, when $\delta \nu +\mu =1$, Lemma~\ref{ex_qa_sequence} and the estimates \eqref{c_eps1609} imply that
 \begin{align}\label{ceps3}
 0< C_{\eps} & \leq  K'\bigg( \frac{2d}{\gamma} D \bigg)^{4n_{t_0, \mathcal{M}_{1}, r'A^2}}\\ \notag
 &  \leq  K'\bigg( \frac{2d}{\gamma} D \bigg)^{4(1-\log t_0)e^{r' A^2}}\\ \notag
 & \leq  K'\bigg( \frac{2d}{\gamma} D \bigg)^{4(1+\log K+2d \log A-\frac{1}{2}\log \eps)e^{r' A^2}}.
 \end{align}
 While setting $D'=4 \max\big(1+\log K, 2d, r', \frac{1}{2}\big)$, we obtain from \eqref{ceps3}
 \begin{equation*}
 0< C_{\eps} \leq  K'\bigg( \frac{2d}{\gamma} D \bigg)^{D'(1+ \log A-\log \eps)e^{D'A^2}}.
 \end{equation*}
 This ends the proof of Theorem~\ref{specific_GS_uncertaintyprinciple}.

\subsection{Proof of Theorem~\ref{uncertainty_principle}}\label{thm_hermite_proof}
Let $0 \leq \delta \leq 1$ and $\Theta : [0,+\infty) \longrightarrow [0,+\infty)$ be a non-negative continuous function such that the associated sequence $(M_p)_{p \in \nn}$ in \eqref{lc_sequence} satisfies the assumptions $\text{(H1)}$, $\text{(H2)}$ and $\text{(H3)}_{\frac{1+\delta}{2}}$. Beforehand, let us notice that the logarithmic convexity property of the sequence $(M_p)_{p \in \nn}$, that is,
\begin{equation*}
\forall p \in \nn^*, \quad M_{p}^2 \leq M_{p+1} M_{p-1},
\end{equation*}
implies that
\begin{equation*}
\forall p \in \nn^*, \quad \frac{M_p}{M_{p+1}} \leq \frac{M_{p-1}}{M_p} \leq \frac{M_0}{M_1},
\end{equation*}
since $M_p >0$ for all $p \in \nn$. It follows that the modified sequence $(M'_p)_{p \in \nn}= \big(\big(\frac{M_0}{M_1}\big)^p M_p \big)_{p \in \nn}$ is a non-decreasing logarithmically convex sequence.
Let $f \in GS_{\Theta}$ and $s=\frac{1+\delta}{2}$. According to Proposition~\ref{bernstein_estim1}, there exists a positive constant $D_{\Theta, d, \delta}\geq 1$ independent on $f$, such that for all $r \geq 0$, $\beta \in \nn^d$,
\begin{align*}
\|\langle x \rangle^{r} \partial_x^{\beta} f \|_{L^2(\rr^d)} & \leq (D_{\Theta,d,\delta})^{1+r+|\beta|} \Big(M_{\left\lfloor \frac{r+1 +|\beta|+(2-s)(d+1)}{2s} \right\rfloor +1}\Big)^s \|f\|_{GS_{\Theta}} \\
& =(D_{\Theta,d,\delta})^{1+r+|\beta|} \Big(\frac{M_1}{M_0}\Big)^{s\left\lfloor \frac{r+1 +|\beta|+(2-s)(d+1)}{2s} \right\rfloor+s} \Big(M'_{\left\lfloor \frac{r+1 +|\beta|+(2-s)(d+1)}{2s} \right\rfloor +1}\Big)^s \|f\|_{GS_{\Theta}} \\
& \leq D'^{1+r+|\beta|} \Big(M'_{\left\lfloor \frac{r+1 +|\beta|+(2-s)(d+1)}{2s} \right\rfloor +1}\Big)^s \|f\|_{GS_{\Theta}},
\end{align*}
where $D'=D'( \Theta, d , \delta) \geq 1$ is a new positive constant. Since the sequence $(M'_p)_{p \in \nn}$ is non-decreasing and $\frac{1}{2} \leq s \leq 1$, we deduce that 
\begin{equation*}
\forall r \geq0, \forall \beta \in \nn^d, \quad \|\langle x \rangle^{r} \partial_x^{\beta} f \|_{L^2(\rr^d)} \leq D'^{1+r+|\beta|}  \Big(M'_{\left\lfloor \frac{r+|\beta|}{2s} \right\rfloor +2d+4}\Big)^s \|f\|_{GS_{\Theta}}.
\end{equation*}
This implies that $f \in GS_{\mathcal{N}, \rho}$ and 
$$\| f \|_{GS_{\mathcal{N}, \rho}} \leq \|f\|_{GS_{\Theta}}$$ with $$\mathcal{N}=(N_{p,q})_{p,q \in \nn^d}=\Big(\max(R,1)^pD'^{1+p+q}  \Big(M'_{\left\lfloor \frac{\delta p+q}{1+\delta} \right\rfloor +2d+4}\Big)^{\frac{1+\delta}{2}}\Big)_{(p,q) \in \nn^2}.$$
We conclude by applying Theorem~\ref{general_uncertaintyprinciple}. To that end, we notice that the non-decreasing property of the sequence $(M'_p)_{p \in \nn}$ ensures that the sequence $\mathcal{N}$ is non-decreasing with respect to the two indexes. We deduce from the Denjoy-Carleman Theorem (Theorem~\ref{Den_Carl_thm}) and assumption $\text{(H3)}_{\frac{1+\delta}{2}}$ that the diagonal sequence $$(N_{p,p})_{p \in \nn} = \big(\max(R,1)^pD'^{1+2p}  (M'_{p +2d+4})^{\frac{1+\delta}{2}}\big)_{p \in \nn},$$ is quasi-analytic since
\begin{multline*}
\sum_{p=0}^{+\infty} \frac{N_{p,p}}{N_{p+1,p+1}}= \frac{1}{ \max(R,1)D'^2} \sum_{p=2d+4}^{+\infty} \Big(\frac{{M'_p}}{{M'_{p+1}}}\Big)^{\frac{1+\delta}{2}} \\= \frac{1}{\max(R,1) D'^2} \Big(\frac{M_1}{M_0}\Big)^{\frac{1+\delta}{2}} \sum_{p=2d+4}^{+\infty} \Big(\frac{{M_p}}{{M_{p+1}}}\Big)^{\frac{1+\delta}{2}}=+\infty.
\end{multline*}
The result of Theorem~\ref{uncertainty_principle} then follows from Theorem~\ref{general_uncertaintyprinciple}. It ends the proof of Theorem~\ref{uncertainty_principle}.
\section{Proof of Theorem~\ref{observability_result}}\label{proof_obs}
This section is devoted to the proof of the null-controllability result given by Theorem~\ref{observability_result}.

Let $(A,D(A))$ be a closed operator on $L^2(\rr^d)$ which is the  infinitesimal generator of a strongly continuous contraction semigroup $(e^{-tA})_{t \geq 0}$ on $L^2(\rr^d)$ that satisfies the following quantitative smoothing estimates: there exist some constants $0 < \mu <1$, $\nu > 0$ with $\mu+ \nu \geq 1$ and 
 $C \geq 1$, $r_1>0$, $r_2 \geq 0$, $0< t_0 \leq 1$ such that
\begin{multline}\label{GS_estimate2}
\forall 0< t \leq t_0, \forall \alpha, \beta \in \nn^d, \forall g \in L^2(\rr^d),\\
\| x^{\alpha}\partial_x^{\beta}( e^{-t A^*}g)\|_{L^2(\rr^d)} \leq  \frac{C ^{1+|\alpha|+|\beta|}}{t^{r_1 (|\alpha|+|\beta|)+r_2}}(\alpha !)^{\nu} (\beta!)^{\mu} \|g\|_{L^2(\rr^d)},
\end{multline}
where $A^*$ denotes the $L^2(\rr^d)$-adjoint of $A$.
Let $\rho : \rr^d \longrightarrow (0,+\infty)$ be a $L$-Lipschitz positive function
with~$\rr^d$ being equipped with the Euclidean norm and $0 < L <1$
such that there exist some constants $0 \leq  \delta < \frac{1-\mu}{\nu}$, $m>0$, $R>0$ so that
\begin{equation*}
\forall x \in \rr^d, \quad 0<m \leq \rho(x) \leq R{\left\langle x\right\rangle}^{\delta}.
\end{equation*} 
Let $\omega \subset \rr^d$ be a measurable subset that is thick with respect to the density $\rho$. Let us show that
Theorem~\ref{observability_result} can be deduced from the uncertainty principles given in Theorem~\ref{specific_GS_uncertaintyprinciple}. To that end, we deduce from the estimates \eqref{GS_estimate2} and Lemma~\ref{croch} that there exists a positive constant $C'=C'(C, d)\geq 1$ such that
\begin{multline}\label{bernstein_estimate_GS}
\forall 0< t \leq t_0, \forall p \in \nn, \forall \beta \in \nn^d, \forall g \in L^2(\rr^d),\\
\| \langle x \rangle^{p} \partial_x^{\beta}( e^{-t A^*}g)\|_{L^2(\rr^d)} \leq  \frac{C'^{1+p+|\beta|}}{t^{r_1 (p+|\beta|)+r_2}}(p !)^{\nu} (|\beta|!)^{\mu} \|g\|_{L^2(\rr^d)}.
\end{multline}
It follows from \eqref{bernstein_estimate_GS} and Theorem~\ref{specific_GS_uncertaintyprinciple} applied with $f=e^{-tA^*}g \in \mathscr{S}(\rr^d)$ that there exists a positive constant $K=K(\gamma, d, \rho, \mu, \nu) \geq 1$ such that
$\forall 0< t \leq t_0, \forall g \in L^2(\rr^d), \forall 0< \eps \leq 1,$
\begin{equation}
\|e^{-tA^*}g\|^2_{L^2(\rr^d)} \leq e^{K ( 1-\log \eps+(C' t^{-r_1})^{\frac{2}{1-s}})} \|e^{-tA^*}g\|^2_{L^2(\omega)} + \frac{C'^2}{t^{2r_2}} \eps \|g\|^2_{L^2(\rr^d)},
\end{equation} 
with $0<s=\delta \nu +\mu<1$.
Thanks to the contraction property of the semigroup $(e^{-tA^*})_{t \geq 0}$, we deduce that for all $0< \tau \leq t_0$, $\frac{1}{2} \leq q <1$, $0<\eps \leq 1$, $g \in L^2(\rr^d)$,
\begin{align*}
\|e^{-\tau A^*}g\|^2_{L^2(\rr^d)} & \leq \frac{1}{(1-q)\tau} \int_{q\tau}^{\tau} \|e^{-t A^*}g \|^2_{L^2(\rr^d)} dt \\ \nonumber
& \leq \frac{ e^{K (1-\log\eps+(C'(q\tau)^{-r_1})^{\frac{2}{1-s}})}}{(1-q)\tau} \int_{q \tau}^{\tau} \|e^{-tA^*} g \|^2_{L^2(\omega)} dt + \eps \frac{C'^2}{(q \tau)^{2r_2}} \|g\|^2_{L^2(\rr^d)} \\ \nonumber
& \leq \frac{ e^{K (1-\log\eps+(C'2^{r_1}\tau^{-r_1})^{\frac{2}{1-s}})}}{(1-q)\tau} \int_{q \tau}^{\tau} \|e^{-tA^*} g \|^2_{L^2(\omega)} dt + \eps \frac{4^{r_2}C'^2}{\tau^{2 r_2}} \|g\|^2_{L^2(\rr^d)}.
\end{align*}
For $0< \tau \leq t_0$ and $\frac{1}{2} \leq q <1$, we choose $$0<\eps = \exp\big(-\tau^{-\frac{2r_1}{1-s}}\big)\leq 1.$$ Since $1 \leq \frac{1}{\tau^{2r_1}}$, it follows that there exists a new constant $K'=K'(\gamma, d, \rho, \delta, \mu, \nu,C', r_1, s)\geq 1$ such that for all $0< \tau \leq t_0$, $\frac{1}{2} \leq q <1$, $g \in L^2(\rr^d)$,
\begin{equation*}
\|e^{-\tau A^*}g\|^2_{L^2(\rr^d)} \leq \frac{e^{K'\tau^{-\frac{2r_1}{1-s}}}}{(1-q)\tau} \int_{q \tau}^{\tau} \|e^{-tA^*} g \|^2_{L^2(\omega)} dt + \exp\big(-\tau^{-\frac{2r_1}{1-s}}\big) \frac{4^{r_2}C'^2}{\tau^{2 r_2}} \|g\|^2_{L^2(\rr^d)}.
\end{equation*}
We follow the strategy developed by Miller in \cite{Miller}.
Let $0<t_1\leq t_0$ such that for all $0< \tau \leq t_1$,
\begin{equation*}
\frac{\exp \Big(K' \tau^{-\frac{2r_1}{1-s}}\Big)}{\tau}  \leq \exp\Big(2K' \tau^{-\frac{2r_1}{1-s}}\Big)
\end{equation*}
and 
\begin{equation*}
\exp\big(-\tau^{-\frac{2r_1}{1-s}}\big) \frac{4^{r_2}C'^2}{\tau^{2 r_2}} \leq \exp\Big(-\frac{\tau^{-\frac{2r_1}{1-s}}}{2}\Big).
\end{equation*}
It follows that for all $0 < \tau \leq t_1$, $\frac{1}{2} \leq q <1$, $g \in L^2(\rr^d)$,
\begin{multline*}
 (1-q)\exp\Big(-\frac{2K'}{\tau^{\frac{2r_1}{1-s}}}\Big) \|e^{-\tau A^*}g\|^2_{L^2(\rr^d)} \\
 \leq \int_{q \tau}^{\tau} \|e^{-tA^*} g \|^2_{L^2(\omega)} dt +  (1-q)\exp\Big(-\frac{2K'+\frac{1}{2}}{\tau^{\frac{2r_1}{1-s}}}\Big)\|g\|^2_{L^2(\rr^d)}.
\end{multline*}
Setting $f(\tau)=(1-q)\exp\Big(-\frac{2K'}{\tau^{\frac{2r_1}{1-s}}}\Big)$ and choosing $q$ so that $$ \max\Big(\Big(\frac{2K'}{2K'+\frac{1}{2}}\Big)^{\frac{1-s}{2r_1}}, \frac{1}{2} \Big) \leq q<1,$$ we obtain that for all $0< \tau \leq t_1$ and $g \in L^2(\rr^d)$,
\begin{equation}\label{approx_obs1609}
f(\tau) \|e^{-\tau A^*}g\|^2_{L^2(\rr^d)} \leq \int_{q \tau}^{\tau} \|e^{-tA^*} g \|^2_{L^2(\omega)} dt +  f(q \tau) \|g\|^2_{L^2(\rr^d)}.
\end{equation}
Thanks to this estimate, the observability estimate is established as follows: let $0< T \leq t_1$ and define the two sequences $(\tau_k)_{k \geq 0}$ and $(T_k)_{k \geq 0}$ as
\begin{equation*}
\forall k \geq 0, \quad \tau_k= q^k (1-q) T \quad \text{ and } \quad  \forall k \geq 0, \quad  T_{k+1}= T_k-\tau_k, \quad T_0=T.
\end{equation*}
By applying \eqref{approx_obs1609} with $e^{-T_{k+1}A^*}g$, it follows that for all $g \in L^2(\rr^d)$ and $k \in \nn$,
\begin{multline*}
f(\tau_k) \|e^{-T_k A^*}g\|^2_{L^2(\rr^d)} -f(\tau_{k+1}) \|e^{-T_{k+1} A^*}g\|^2_{L^2(\rr^d)} \\
\leq \int_{\tau_{k+1}}^{\tau_k} \|e^{-(t+T_{k+1})A^*}g \|^2_{L^2(\omega)} dt 
= \int_{\tau_{k+1}+T_{k+1}}^{T_k} \|e^{-tA^*}g \|^2_{L^2(\omega)} dt \leq \int_{T_{k+1}}^{T_k} \|e^{-tA^*}g \|^2_{L^2(\omega)} dt.
\end{multline*}
By summing over all the integers $k \in \nn$ and by noticing that $$\lim \limits_{k \to +\infty} f(\tau_k)=0, \quad \lim \limits_{k \to +\infty} T_k= T- \sum_{k \in \nn} \tau_k =0,$$
and $$ \forall k \geq 0, \quad \|e^{-T_k A^*}g\|_{L^2(\rr^d)} \leq \|g\|_{L^2(\rr^d)},$$ by the contraction property of the semigroup $(e^{-tA^*})_{t \geq 0}$, 
it follows that
\begin{equation*}
\|e^{-T A^*}g\|^2_{L^2(\rr^d)} \leq \frac{1}{1-q}\exp\Big(\frac{2K'}{((1-q)T)^{\frac{2r_1}{1-\mu -\delta\nu}}} \Big) \int_{0}^{T} \|e^{-tA^*}g \|^2_{L^2(\omega)} dt.
\end{equation*}
By using anew the contraction property of the semigroup $(e^{-tA^*})_{t \geq 0}$, we deduce that for all $g \in L^2(\rr^d)$, $T \geq t_1$,
\begin{multline*}
\|e^{-T A^*}g\|^2_{L^2(\rr^d)} \leq \|e^{-t_1 A^*}g\|^2_{L^2(\rr^d)} \leq \frac{1}{1-q} \exp\Big(\frac{2K'}{((1-q)t_1)^{\frac{2 r_1}{1-\mu -\delta\nu}}} \Big) \int_{0}^{t_1} \|e^{-tA^*}g \|^2_{L^2(\omega)} dt \\
\leq \frac{1}{1-q} \exp\Big(\frac{2K'}{((1-q)t_1)^{\frac{2 r_1}{1-\mu -\delta\nu}}} \Big) \int_{0}^{T} \|e^{-tA^*}g \|^2_{L^2(\omega)} dt.
\end{multline*}
This ends the proof of Theorem~\ref{observability_result}.

\section{Appendix}\label{appendix}
\subsection{Bernstein type estimates}\label{Hermite_functions}
This section is devoted to the proof of the Bernstein type estimates given in Proposition~\ref{bernstein_estim1}. To that end, we begin by recalling basic facts about Hermite functions. The standard Hermite functions $(\phi_{k})_{k\geq 0}$ are defined for $x \in \rr$,
 \begin{equation*}\label{defi}
 \phi_{k}(x)=\frac{(-1)^k}{\sqrt{2^k k!\sqrt{\pi}}} e^{\frac{x^2}{2}}\frac{d^k}{dx^k}(e^{-x^2})
 =\frac{1}{\sqrt{2^k k!\sqrt{\pi}}} \Bigl(x-\frac{d}{dx}\Bigr)^k(e^{-\frac{x^2}{2}})=\frac{ a_{+}^k \phi_{0}}{\sqrt{k!}},
\end{equation*}
where $a_{+}$ is the creation operator
$$a_{+}=\frac{1}{\sqrt{2}}\Big(x-\frac{d}{dx}\Big).$$
The Hermite functions satisfy the identity 
\begin{equation*}\label{eq2ui1}
\forall k \in \nn, \quad \Big(-\frac{d^2}{dx^2}+x^2\Big)\phi_{k}=(2k+1)\phi_{k}.
\end{equation*}
The family $(\phi_{k})_{k\in \nn}$ is a Hilbert basis of $L^2(\R)$.
We set for $\alpha=(\alpha_{j})_{1\le j\le d}\in\N^d$, $x=(x_{j})_{1\le j\le d}\in \R^d,$
\begin{equation*}\label{jk1}
\Phi_{\alpha}(x)=\prod_{j=1}^d\phi_{\alpha_j}(x_j).
\end{equation*}
The family $(\Phi_{\alpha})_{\alpha \in \nn^d}$ is a Hilbert basis of $L^2(\R^d)$
composed of the eigenfunctions of the $d$-dimensional harmonic oscillator
\begin{equation*}\label{6.harmo}
\mathcal{H}=-\Delta_x+|x|^2=\sum_{k\ge 0}(2k+d)\mathbb P_{k},\quad \text{Id}=\sum_{k \ge 0}\mathbb P_{k},
\end{equation*}
where $\mathbb P_{k}$ is the orthogonal projection onto $\text{Span}_{\cc}
\{\Phi_{\alpha}\}_{\alpha\in \N^d,\val \alpha =k}$, with $\val \alpha=\alpha_{1}+\dots+\alpha_{d}$. 

Instrumental in the sequel are the following basic estimates proved by Beauchard, Jaming and Pravda-Starov in the proof of \cite[Proposition~3.3]{kkj} (formula (3.38)). 
\begin{lemma}\label{lem1}
With $\mathcal{E}_N= \emph{\textrm{Span}}_{\cc}\{\Phi_{\alpha}\}_{\alpha \in \nn^d, \ |\alpha| \leq N}$, we have for all $N \in \nn$, $f \in \mathcal{E}_N$,
\begin{equation*}
\forall (\alpha, \beta) \in \nn^d \times \nn^d, \quad
\|x^{\alpha}\partial_x^{\beta}f\|_{L^2(\rr^d)}\leq 2^{\frac{|\alpha|+|\beta|}{2}}\sqrt{\frac{(N+|\alpha|+|\beta|)!}{N!}}\|f\|_{L^2(\rr^d)}.
\end{equation*}
\end{lemma}

We can now prove Proposition~\ref{bernstein_estim1}. Let $\Theta : [0,+\infty) \longrightarrow [0,+\infty)$ be a non-negative continuous function such that the associated sequence $(M_p)_{p \in \nn}$ in \eqref{lc_sequence} satisfies the assumptions $(H1)$ and $(H2)$. Let $f \in GS_{\Theta}$, $0< s \leq 1$ and $(\alpha, \beta) \in \nn^d \times \nn^d$. We begin by proving that there exist some positive constants $C'_{\Theta}>0$, $\tilde{C}_{\Theta}>0$, independent on $f$, $\alpha$ and $\beta$ such that
\begin{equation*}\label{gse_1709}
\| x^{\alpha} \partial_x^{\beta} f \|_{L^2(\rr^d)} \leq C'_{\Theta} \tilde{C}^{|\alpha|+|\beta|}_{\Theta}  \Big(M_{\left\lfloor \frac{|\alpha| +|\beta|+(2-s)(d+1)}{2s} \right\rfloor +1}\Big)^s \|f\|_{GS_{\Theta}}.
\end{equation*}
It is sufficient to prove that there exist some positive constants $C'_{\Theta}>0$, $\tilde{C}_{\Theta}>0$, independent on $f$, $\alpha$ and $\beta$ such that for all $N \geq |\alpha|+|\beta|+1$,
\begin{equation}\label{bernst_goal}
\|x^{\alpha} \partial_x^{\beta} \pi_Nf \|_{L^2(\rr^d)} \leq C'_{\Theta} \tilde{C}^{|\alpha|+|\beta|}_{\Theta}  \Big(M_{\left\lfloor \frac{|\alpha| +|\beta|+(2-s)(d+1)}{2s} \right\rfloor +1}\Big)^s \|f\|_{GS_{\Theta}},
\end{equation}
with $\pi_N f$ the orthogonal projection of the function $f$ onto the space $\textrm{Span}_{\cc}\{\Phi_{\alpha}\}_{\alpha \in \nn^d, \ |\alpha| \leq N}$ given by
\begin{equation}\label{orth_proj}
\pi_N f = \sum_{\substack{\alpha \in \nn^d, \\ |\alpha| \leq N}} \left\langle f, \Phi_{\alpha} \right\rangle_{L^2(\rr^d)} \Phi_{\alpha}.
\end{equation}
Indeed, by using that $(\pi_N f)_{N \in \nn}$ converges to $f$ in $L^2(\rr^d)$ and therefore in $\mathcal{D}'(\rr^d)$, we obtain that the sequence $(x^{\alpha} \partial^{\beta}_x \pi_{N} f)_{N\in \nn}$ converges to $x^{\alpha}\partial^{\beta}_x f$ in $\mathcal{D}'(\rr^d)$. If the estimates \eqref{bernst_goal} hold, the sequence $(x^{\alpha} \partial^{\beta}_x \pi_{N} f)_{N\in \nn}$ is bounded in $L^2(\rr^d)$ and therefore weakly converges (up to a subsequence) to a limit $g \in L^2(\rr^d)$. Thanks to the uniqueness of the limit in $\mathcal{D}'(\rr^d)$, it follows that $g=x^{\alpha} \partial^{\beta}_x f \in L^2(\rr^d)$. Moreover, we have
\begin{equation*}
\|x^{\alpha} \partial_x^{\beta} f \|_{L^2(\rr^d)} \leq \liminf_{N \to +\infty} \|x^{\alpha} \partial_x^{\beta} \pi_{\phi(N)} f \|_{L^2(\rr^d)} \leq C'_{\Theta} \tilde{C}^{|\alpha|+|\beta|}_{\Theta}  \Big(M_{\left\lfloor \frac{|\alpha| +|\beta|+(2-s)(d+1)}{2s} \right\rfloor +1}\Big)^s\|f\|_{GS_{\Theta}}. 
\end{equation*}
Let us prove that the estimates \eqref{bernst_goal} hold. Since $\pi_{|\alpha|+|\beta|}$ is an orthogonal projection on $L^2(\rr^d)$ and therefore satisfies $$\| \pi_{|\alpha|+|\beta|} f \|_{L^2(\rr^d)} \leq \| f \|_{L^2(\rr^d)},$$ we deduce from Lemma~\ref{lem1} and \eqref{orth_proj} that for all $N \geq |\alpha|+ |\beta|+1$,
\begin{align*}
& \quad  \|x^{\alpha}\partial_x^{\beta} \pi_N f \|_{L^2} \leq  \|x^{\alpha}\partial_x^{\beta} \pi_{|\alpha|+|\beta|} f \|_{L^2}  + \|x^{\alpha}\partial_x^{\beta} (\pi_{N}-\pi_{|\alpha|+|\beta|}) f \|_{L^2}   \\ \notag
& \leq 2^{\frac{|\alpha|+|\beta|}{2}}\sqrt{\frac{(2(|\alpha|+|\beta|))!}{(|\alpha|+|\beta|)!}} \|f\|_{L^2(\rr^d)}  + \sum_{\substack{\gamma \in \nn^d, \\ |\alpha|+|\beta|+1 \leq |\gamma| \leq N}} |\left\langle f, \Phi_{\gamma} \right\rangle_{L^2}| \|x^{\alpha} \partial_x^{\beta} \Phi_{\gamma} \|_{L^2(\rr^d)}.
\end{align*}
By using anew Lemma~\ref{lem1}, it follows that for all $N \geq |\alpha|+ |\beta|+1$,
\begin{align}\label{gse_1709}
& \|x^{\alpha}\partial_x^{\beta} \pi_N f \|_{L^2} \\ \notag 
& \leq 2^{|\alpha|+|\beta|}(|\alpha|+|\beta|)^{\frac{|\alpha|+|\beta|}{2}} \|f\|_{L^2(\rr^d)}  + \sum_{\substack{\gamma \in \nn^d, \\ |\alpha|+|\beta|+1 \leq |\gamma| \leq N}} |\left\langle f, \Phi_{\gamma} \right\rangle_{L^2}|  2^{\frac{|\alpha|+|\beta|}{2}}\sqrt{\frac{(|\gamma|+|\alpha|+|\beta|)!}{|\gamma|!}} \\ \notag
& \leq 2^{|\alpha|+|\beta|}(|\alpha|+|\beta|)^{\frac{|\alpha|+|\beta|}{2}} \|f\|_{L^2(\rr^d)}  + \sum_{\substack{\gamma \in \nn^d, \\ |\alpha|+|\beta|+1 \leq |\gamma| \leq N}} |\left\langle f, \Phi_{\gamma} \right\rangle_{L^2}|  2^{|\alpha|+|\beta|} |\gamma|^{\frac{|\alpha|+|\beta|}{2}}.
\end{align}
On the first hand, it follows from $0< s \leq 1$ that for all $N \geq |\alpha|+ |\beta|+1$,
\begin{align}\label{gse_2}
& \quad \sum_{\substack{\gamma \in \nn^d, \\ |\alpha|+|\beta|+1 \leq |\gamma| \leq N}} |\left\langle f, \Phi_{\gamma} \right\rangle_{L^2}| |\gamma|^{\frac{|\alpha|+|\beta|}{2}} \\ \notag
&\leq \sum_{\substack{\gamma \in \nn^d, \\ |\gamma| \geq 1}} |\left\langle f, \Phi_{\gamma} \right\rangle_{L^2}| e^{s\Theta(|\gamma|)} |\gamma|^{-\frac{(2-s)(d+1)}{2}} |\gamma|^{\frac{|\alpha|+|\beta|+(2-s)(d+1)}{2}}e^{-s\Theta(|\gamma|)} \\ \notag
&\leq \Big(M_{\left\lfloor \frac{|\alpha| +|\beta|+(2-s)(d+1)}{2s} \right\rfloor +1}\Big)^s \sum_{\substack{\gamma \in \nn^d, \\ |\gamma| \geq 1}} |\left\langle f, \Phi_{\gamma} \right\rangle_{L^2}| e^{s\Theta(|\gamma|)} |\gamma|^{-\frac{(2-s)(d+1)}{2}} \\ \notag
&\leq \Big(M_{\left\lfloor \frac{|\alpha| +|\beta|+(2-s)(d+1)}{2s} \right\rfloor +1}\Big)^s \vert|\big(\left\langle f, \Phi_{\gamma} \right\rangle_{L^2}\big)_{\gamma \in \nn^d} \vert|_{l^{\infty}(\nn^d)}^{1-s} \sum_{\substack{\gamma \in \nn^d, \\ |\gamma| \geq 1}} \Big(|\left\langle f, \Phi_{\gamma} \right\rangle_{L^2}| e^{\Theta(|\gamma|)}\Big)^s |\gamma|^{-\frac{(2-s)(d+1)}{2}}.
\end{align}
H\"older's inequality implies that for all $0 < s \leq 1$,
\begin{equation}\label{holder}
 \sum_{\substack{\gamma \in \nn^d, \\ |\gamma| \geq 1}} \Big(|\left\langle f, \Phi_{\gamma} \right\rangle_{L^2}| e^{\Theta(|\gamma|)}\Big)^s |\gamma|^{-\frac{(2-s)(d+1)}{2}} \leq D_{d,s}\left\|\Big(e^{\Theta(|\gamma|)}\left\langle f, \Phi_{\gamma} \right\rangle_{L^2}\Big)_{\gamma \in \nn^d} \right\|_{l^{2}(\nn^d)}^{s},
\end{equation}
 with 
\begin{equation*}
D_{d,s} =\Big(\sum_{\substack{\gamma \in \nn^d, \\ |\gamma| \geq 1}} |\gamma|^{-(d+1)}\Big)^{1-\frac{s}{2}} < +\infty.
\end{equation*}
Since $\Theta(|\gamma|) \geq 0$ for all $\gamma \in \nn^d$, it follows that
\begin{equation}\label{infnorm1709}
\vert|\big(\left\langle f, \Phi_{\gamma} \right\rangle_{L^2}\big)_{\gamma \in \nn^d} \vert|_{l^{\infty}(\nn^d)} \leq \vert|\big(\left\langle f, \Phi_{\gamma} \right\rangle_{L^2} e^{\Theta(|\gamma|)}\big)_{\gamma \in \nn^d} \vert|_{l^{\infty}(\nn^d)} \leq \| f \|_{GS_{\Theta}}.
\end{equation}
We deduce from \eqref{gse_2}, \eqref{holder} and \eqref{infnorm1709} that
\begin{equation}\label{gse_3}
\sum_{\substack{\gamma \in \nn^d, \\ |\alpha|+|\beta|+1 \leq |\gamma| \leq N}} |\left\langle f, \Phi_{\gamma} \right\rangle_{L^2}| |\gamma|^{\frac{|\alpha|+|\beta|}{2}}
 \leq D_{d,s} \Big(M_{\left\lfloor \frac{|\alpha| +|\beta|+(2-s)(d+1)}{2s} \right\rfloor +1}\Big)^s \|f\|_{GS_{\Theta}}.
 \end{equation}
On the other hand, the assumption $\text{(H2)}$ implies that there exist $C_{\Theta}\geq 1$ and $L_{\Theta}\geq 1$ such that if $|\alpha|+|\beta| \geq 1$ then
\begin{align}\label{gse_4}
(|\alpha|+|\beta|)^{\frac{|\alpha|+|\beta|}{2}} & = (2s)^{\frac{|\alpha|+|\beta|}{2}} \Big(\frac{|\alpha|+|\beta|}{2s} \Big)^{s \frac{|\alpha|+|\beta|}{2s}} \\ \notag
& \leq (2s)^{\frac{|\alpha|+|\beta|}{2}} \Big(\left\lfloor \frac{|\alpha|+|\beta|}{2s} \right\rfloor +1\Big)^{s \Big(\left\lfloor \frac{|\alpha|+|\beta|}{2s}\right\rfloor +1\Big)} \\ \notag
& \leq (2s)^{\frac{|\alpha|+|\beta|}{2}} C^s_{\Theta}L^s_{\Theta} L_{\Theta}^{\frac{|\alpha|+|\beta|}{2}} \Big(M_{\left\lfloor \frac{|\alpha| +|\beta|}{2s} \right\rfloor +1}\Big)^s.
\end{align}
The last inequality holds true as well when $|\alpha| +|\beta|=0$ with the convention $0^0=1$ since $C_{\Theta} M_1 \geq 1$ and $L_{\Theta} \geq 1$.
The logarithmical convexity of the sequence $(M_p)_{p \in \nn}$ gives
\begin{equation*}\label{log_conv0}
\forall p \in \nn, \quad M_p \leq \frac{M_{0}}{M_1} M_{p+1}
\end{equation*}
and therefore,
\begin{equation*}\label{log_conv}
\forall 0 \leq p \leq q, \quad M_p \leq \Big(\frac{M_{0}}{M_1}\Big)^{q-p} M_{q}.
\end{equation*}
By using this estimate together with the following elementary inequality 
\begin{equation*}
\forall x,y \geq 0, \quad \lfloor x +y \rfloor \leq \lfloor x \rfloor + \lfloor y \rfloor +1,
\end{equation*}
we obtain
\begin{equation}\label{log_conv2}
\forall 0 \leq r \leq r', \quad M_{\lfloor r \rfloor} \leq \max\Big(1,\frac{M_{0}}{M_1}\Big)^{\lfloor r'-r \rfloor+1} M_{\lfloor r' \rfloor}.
\end{equation}
It follows from \eqref{gse_4} and \eqref{log_conv2} that
\begin{align}\label{gse_5}
(|\alpha|+|\beta|)^{\frac{|\alpha|+|\beta|}{2}}
& \leq C_{\Theta}^s L_{\Theta}^s \Big(\sqrt{2s L_{\Theta}} \Big)^{|\alpha|+|\beta|} \max\Big(1,\frac{M_0}{M_1}\Big)^{s (\left\lfloor \frac{(2-s)(d+1)}{2s} \right\rfloor +1)}\Big(M_{\left\lfloor \frac{|\alpha| +|\beta|+(2-s)(d+1)}{2s} \right\rfloor +1}\Big)^s \\ \notag
& \leq  C_{\Theta}^s L_{\Theta}^s \Big(\sqrt{2s L_{\Theta}} \Big)^{|\alpha|+|\beta|} \max\Big(1,\frac{M_0}{M_1}\Big)^{d+2} \Big(M_{\left\lfloor \frac{|\alpha| +|\beta|+(2-s)(d+1)}{2s} \right\rfloor +1}\Big)^s,
\end{align}
since $0< s \leq 1$.
We deduce from \eqref{gse_1709}, \eqref{gse_3} and \eqref{gse_5} that for all $N \geq |\alpha|+|\beta|+1$,

\begin{equation*}
\| x^{\alpha} \partial_x^{\beta} \pi_N f \|_{L^2(\rr^d)} \leq K_{\Theta, s} K'^{|\alpha|+|\beta|}_{\Theta,s} \Big(M_{\left\lfloor \frac{|\alpha| +|\beta|+(2-s)(d+1)}{2s} \right\rfloor +1}\Big)^s \|f\|_{GS_{\Theta}},
\end{equation*}
with $K_{\Theta,s} = D_{d,s} +C^s_{\Theta} L^s_{\Theta} \max\Big(1,\frac{M_0}{M_1}\Big)^{d+2}\geq 1$ and $K'_{\Theta,s} = 2 \max(1, \sqrt{2sL_{\Theta}}) \geq 1$.
This implies that $f \in \mathscr{S}(\rr^d)$ and for all $\alpha, \beta \in \nn^d$,
\begin{equation}\label{gse_6}
\| x^{\alpha} \partial_x^{\beta} f \|_{L^2(\rr^d)} \leq K_{\Theta, s} K'^{|\alpha|+|\beta|}_{\Theta,s}  \Big(M_{\left\lfloor \frac{|\alpha| +|\beta|+(2-s)(d+1)}{2s} \right\rfloor +1}\Big)^s \|f\|_{GS_{\Theta}}.
\end{equation}
By using Newton formula, we obtain that for all $k \in \nn$,
\begin{multline*}
\|\left\langle x\right\rangle^k \partial_x^\beta f \|_{L^2(\rr^d)}^2 = \int_{\rr^d} \Big( 1 + \sum \limits_{i=1}^d {x_i^2} \Big)^k |\partial_x^\beta f(x) |^2 dx \\ = \int_{\rr^d} \sum_{\substack{\gamma \in \nn^{d+1}, \\ |\gamma|=k}} \frac{k!}{\gamma !} x^{2 \tilde{\gamma}} |\partial_x^\beta f(x) |^2 dx 
=\sum_{\substack{\gamma \in \nn^{d+1}, \\ |\gamma|=k}} \frac{k!}{\gamma !} \|x^{\tilde{\gamma}} \partial_x^\beta f \|_{L^2(\rr^d)}^2 ,
\end{multline*}
where we denote $\tilde{\gamma}=(\gamma_1,...,\gamma_d) \in \nn^d$ if $\gamma=(\gamma_1,...\gamma_{d+1}) \in \nn^{d+1}$. It follows from \eqref{log_conv2} and \eqref{gse_6} that for all $k \in \nn$ and $\beta \in \nn^d$,
\begin{align}\label{gse7}
\|\left\langle x\right\rangle^k \partial_x^\beta f \|_{L^2(\rr^d)}^2 \leq & \ \sum_{\substack{\gamma \in \nn^{d+1}, \\ |\gamma|=k}} \frac{k!}{\gamma !} K^2_{\Theta,s} K'^{2(|\tilde{\gamma}|+|\beta|)}_{\Theta, s} \Big(M_{\left\lfloor \frac{|\tilde{\gamma}| +|\beta|+(2-s)(d+1)}{2s} \right\rfloor +1}\Big)^{2s} \|f\|^2_{GS_{\Theta}} \\ \nonumber
\leq & \ \sum_{\substack{\gamma \in \nn^{d+1}, \\ |\gamma|=k}} \frac{k!}{\gamma !} K^2_{\Theta, s} K'^{2(k+|\beta|)}_{\Theta, s} \max\Big(1,\frac{M_0}{M_1}\Big)^{k-|\tilde{\gamma}|+2} \Big(M_{\left\lfloor \frac{k +|\beta|+(2-s)(d+1)}{2s} \right\rfloor +1}\Big)^{2s} \|f\|^2_{GS_{\Theta}} \\ \nonumber
\leq & K^2_{\Theta, s} (d+1)^k \max\Big(1,\frac{M_0}{M_1}\Big)^{k+2}K'^{2(k+|\beta|)}_{\Theta, s} \Big(M_{\left\lfloor \frac{k +|\beta|+(2-s)(d+1)}{2s} \right\rfloor +1}\Big)^{2s} \|f\|^2_{GS_{\Theta}},
\end{align}
since 
\begin{equation*}
\sum_{\substack{\gamma \in \nn^{d+1}, \\ |\gamma|=k}} \frac{k!}{\gamma !}=(d+1)^k,
\end{equation*}
thanks to Newton formula.
Let $r \in \rr_+^* \setminus \nn$. There exist $0 < \theta < 1$ and $k \in \nn$ such that 
\begin{equation*}\label{floor}
r= \theta k + (1- \theta)(k+1).
\end{equation*}
By using Hölder inequality, it follows from \eqref{gse7} that
\begin{multline}\label{holder1709}
\|\left\langle x\right\rangle^r \partial_x^\beta f \|_{L^2(\rr^d)}  \leq \|\langle x\rangle^k \partial_x^\beta f\|_{L^2(\rr^d)}^{\theta}\|\langle x\rangle^{k+1} \partial_x^\beta f\|_{L^2(\rr^d)}^{1-\theta} \\
 \leq K_{\Theta, s} (d+1)^{\frac{r}{2}} \max\Big(1,\frac{M_0}{M_1}\Big)^{\frac{r}{2}+1}K'^{r+|\beta|}_{\Theta, s}\Big(M_{\left\lfloor \frac{k +|\beta|+(2-s)(d+1)}{2s} \right\rfloor +1}\Big)^{s \theta} \Big(M_{\left\lfloor \frac{k+1 +|\beta|+(2-s)(d+1)}{2s} \right\rfloor +1}\Big)^{s(1-\theta)} \|f\|_{GS_{\Theta}}.
\end{multline}
By using anew \eqref{log_conv2}, we have
\begin{equation*}
M_{\left\lfloor \frac{k +|\beta|+(2-s)(d+1)}{2s} \right\rfloor +1} \leq \max\Big(1,\frac{M_0}{M_1}\Big)^{\frac{r+1-k}{2s}+1} M_{\left\lfloor \frac{r+1 +|\beta|+(2-s)(d+1)}{2s} \right\rfloor +1}
\end{equation*} 
and
\begin{equation*}
M_{\left\lfloor \frac{k+1 +|\beta|+(2-s)(d+1)}{2s} \right\rfloor +1} \leq \max\Big(1,\frac{M_0}{M_1}\Big)^{\frac{r-k}{2s}+1} M_{\left\lfloor \frac{r+1 +|\beta|+(2-s)(d+1)}{2s} \right\rfloor +1},
\end{equation*}
since $k \leq r$.
We deduce from \eqref{holder1709} that
\begin{align*}
\|\left\langle x\right\rangle^r \partial_x^\beta f \|_{L^2(\rr^d)} & \leq K_{\Theta, s} \max\Big(1,\frac{M_0}{M_1}\Big)^{\frac{2r+\theta-k}{2}+1+s} (d+1)^{\frac{r}{2}} K'^{r+|\beta|}_{\Theta, s} \Big(M_{\left\lfloor \frac{r+1 +|\beta|+(2-s)(d+1)}{2s} \right\rfloor +1}\Big)^s \|f\|_{GS_{\Theta}} \\
& \leq K_{\Theta, s} \max\Big(1,\frac{M_0}{M_1}\Big)^{\frac{r}{2}+3} (d+1)^{\frac{r}{2}} K'^{r+|\beta|}_{\Theta, s} \Big(M_{\left\lfloor \frac{r+1 +|\beta|+(2-s)(d+1)}{2s} \right\rfloor +1}\Big)^s \|f\|_{GS_{\Theta}}, 
\end{align*}
since $0<s\leq 1$, $k \leq r < k+1$ and $0< \theta <1$.
Let us notice that the above inequality also holds for $r \in \nn$. Indeed, it follows from \eqref{log_conv2} and \eqref{gse7} that
\begin{align*}
\|\left\langle x\right\rangle^k \partial_x^\beta f \|_{L^2(\rr^d)} \leq & K_{\Theta, s} (d+1)^{\frac{k}{2}} \max\Big(1,\frac{M_0}{M_1}\Big)^{\frac{k}{2}+1}K'^{k+|\beta|}_{\Theta, s} \Big(M_{\left\lfloor \frac{k +|\beta|+(2-s)(d+1)}{2s} \right\rfloor +1}\Big)^{s} \|f\|_{GS_{\Theta}} \\
\leq &K_{\Theta, s} (d+1)^{\frac{k}{2}} \max\Big(1,\frac{M_0}{M_1}\Big)^{\frac{k}{2}+1+\frac{1}{2}+1}K'^{k+|\beta|}_{\Theta, s} \Big(M_{\left\lfloor \frac{k+1 +|\beta|+(2-s)(d+1)}{2s} \right\rfloor +1}\Big)^{s} \|f\|_{GS_{\Theta}}\\
\leq &K_{\Theta, s} (d+1)^{\frac{k}{2}} \max\Big(1,\frac{M_0}{M_1}\Big)^{\frac{k}{2}+3}K'^{k+|\beta|}_{\Theta, s} \Big(M_{\left\lfloor \frac{k+1 +|\beta|+(2-s)(d+1)}{2s} \right\rfloor +1}\Big)^{s} \|f\|_{GS_{\Theta}}.
\end{align*}
This ends the proof of Proposition~\ref{bernstein_estim1}.

\subsection{Gelfand-Shilov regularity}\label{gelfand}

We refer the reader to the works~\cite{gelfand_shilov,rodino1,rodino,toft} and the references herein for extensive expositions of the Gelfand-Shilov regularity theory.
The Gelfand-Shilov spaces $S_{\nu}^{\mu}(\rr^d)$, with $\mu,\nu>0$, $\mu+\nu\geq 1$, are defined as the spaces of smooth functions $f \in C^{\infty}(\rr^d)$ satisfying the estimates
$$\exists A,C>0, \quad |\partial_x^{\alpha}f(x)| \leq C A^{|\alpha|}(\alpha !)^{\mu}e^{-\frac{1}{A}|x|^{1/\nu}}, \quad x \in \rr^d, \ \alpha \in \mathbb{N}^d,$$
or, equivalently
$$\exists A,C>0, \quad \sup_{x \in \rr^d}|x^{\beta}\partial_x^{\alpha}f(x)| \leq C A^{|\alpha|+|\beta|}(\alpha !)^{\mu}(\beta !)^{\nu}, \quad \alpha, \beta \in \mathbb{N}^d,$$
with $\alpha!=(\alpha_1!)...(\alpha_d!)$ if $\alpha=(\alpha_1,...,\alpha_d) \in \nn^d$.
These Gelfand-Shilov spaces  $S_{\nu}^{\mu}(\rr^d)$ may also be characterized as the spaces of Schwartz functions $f \in \mathscr{S}(\rr^d)$ satisfying the estimates
$$\exists C>0, \eps>0, \quad |f(x)| \leq C e^{-\eps|x|^{1/\nu}}, \quad x \in \rr^d; \qquad |\widehat{f}(\xi)| \leq C e^{-\eps|\xi|^{1/\mu}}, \quad \xi \in \rr^d.$$
In particular, we notice that Hermite functions belong to the symmetric Gelfand-Shilov space  $S_{1/2}^{1/2}(\rr^d)$. More generally, the symmetric Gelfand-Shilov spaces $S_{\mu}^{\mu}(\rr^d)$, with $\mu \geq 1/2$, can be nicely characterized through the decomposition into the Hermite basis $(\Phi_{\alpha})_{\alpha \in \mathbb{N}^d}$, see e.g. \cite[Proposition~1.2]{toft},
\begin{multline*}
f \in S_{\mu}^{\mu}(\rr^d) \Leftrightarrow f \in L^2(\rr^d), \ \exists t_0>0, \ \big\|\big(\langle f,\Phi_{\alpha}\rangle_{L^2}\exp({t_0|\alpha|^{\frac{1}{2\mu}})}\big)_{\alpha \in \mathbb{N}^d}\big\|_{l^2(\mathbb{N}^d)}<+\infty\\
\Leftrightarrow f \in L^2(\rr^d), \ \exists t_0>0, \ \|e^{t_0\mathcal{H}^{\frac{1}{2\mu}}}f\|_{L^2(\rr^d)}<+\infty,
\end{multline*}
where $\mathcal{H}=-\Delta_x+|x|^2$ stands for the harmonic oscillator.
We end this section by proving two technical lemmas:

\begin{lemma}\label{croch}
Let $\mu, \nu >0$ such that $\mu+\nu \geq 1$, $C>0$ and $A \geq 1$. If $f \in S_{\nu}^{\mu}(\rr^d)$ satisfies
\begin{equation}\label{gs_estim}
\forall \alpha \in \nn^d, \forall \beta \in \nn^d, \quad \| x^{\alpha} \partial_x^{\beta} f \|_{L^2(\rr^d)} \leq C A^{|\alpha|+|\beta|} (\alpha!)^{\nu} (\beta!)^{\mu},
\end{equation}
then, it satisfies 
\begin{equation*}
\forall p \in \nn, \forall \beta \in \nn^d, \quad \|\langle x \rangle^{p} \partial_x^{\beta} f \|_{L^2(\rr^d)} \leq C (d+1)^{\frac{p}{2}}A^{p+|\beta|} (p!)^{\nu} (|\beta|!)^{\mu}.
\end{equation*}
\end{lemma}
\begin{proof}
Let $f \in S_{\nu}^{\mu}(\rr^d)$ satisfying the estimates \eqref{gs_estim}. By using Newton formula, we obtain that for all $p \in \nn$, $\beta \in \nn^d$,
\begin{multline}\label{croch_estim}
\|\langle x \rangle^p \partial_x^{\beta} f \|^2_{L^2(\rr^d)} = \int_{\rr^d} \Big(1+\sum_{i=1}^d x_i^2 \Big)^p |\partial_x^{\beta}f(x)|^2 dx \\
= \int_{\rr^d} \sum_{\substack{\gamma \in \nn^{d+1}, \\ |\gamma|=p}} \frac{p!}{\gamma!} x^{2\tilde{\gamma}} |\partial_x^{\beta}f(x)|^2 dx = \sum_{\substack{\gamma \in \nn^{d+1}, \\ |\gamma|=p}} \frac{p!}{\gamma!} \|x^{\tilde{\gamma}} \partial_x^{\beta} f \|^2_{L^2(\rr^d)},
\end{multline}
where we denote $\tilde{\gamma}=(\gamma_1,...,\gamma_d) \in \nn^d$ if $\gamma=(\gamma_1,...,\gamma_{d+1}) \in \nn^{d+1}$. Since for all $\alpha \in \nn^d$, $\alpha! \leq (|\alpha|)!$, it follows from \eqref{gs_estim} and \eqref{croch_estim} that
\begin{align*}
\|\langle x \rangle^p \partial_x^{\beta} f \|^2_{L^2(\rr^d)} & \leq C^2\sum_{\substack{\gamma \in \nn^{d+1}, \\ |\gamma|=p}} \frac{p!}{\gamma!} A^{2(|\tilde{\gamma}|+|\beta|)} (|\tilde{\gamma}|!)^{2\nu} (|\beta|!)^{2\mu} \\
& \leq C^2 (d+1)^p A^{2(p+|\beta|)} (p!)^{2\nu} (|\beta|!)^{2\mu},
\end{align*}
since $$ \sum_{\substack{\gamma \in \nn^{d+1}, \\ |\gamma|=p}} \frac{p!}{\gamma!}  = (d+1)^p.$$
\end{proof}
\begin{lemma}\label{interpolation}
Let $\mu, \nu >0$ such that $\mu+\nu \geq 1$, $0 \leq \delta \leq 1$, $C>0$ and $A \geq 1$. If $f \in S_{\nu}^{\mu}(\rr^d)$ satisfies
\begin{equation}\label{int}
\forall p \in \nn, \forall \beta \in \nn^d, \quad \|\langle x \rangle^p \partial_x^{\beta} f \|_{L^2(\rr^d)} \leq C A^{p+|\beta|} (p!)^{\nu} (|\beta|!)^{\mu},
\end{equation}
then, it satisfies 
\begin{equation*}
\forall p \in \nn, \forall \beta \in \nn^d, \quad \|\langle x \rangle^{\delta p} \partial_x^{\beta} f \|_{L^2(\rr^d)} \leq C(8^{\nu} e^{\nu}A)^{p+|\beta|} (p!)^{\delta \nu} (|\beta|!)^{\mu}.
\end{equation*}
\end{lemma}
\begin{proof}
Let $f \in S_{\nu}^{\mu}(\rr^d)$ satisfying the estimates \eqref{int}. 
It follows from H\"older inequality that for all $r \in (0,+\infty) \setminus \nn$ and $\beta \in \nn^d$,
\begin{multline}\label{holder0}
\|\langle x \rangle^r \partial_x^{\beta} f \|^2_{L^2(\rr^d)} = \int_{\rr^d} \big(\langle x \rangle^{2 \lfloor r \rfloor} |\partial^{\beta}_x f(x)|^2\big)^{\lfloor r \rfloor +1-r} \big(\langle x \rangle^{2(\lfloor r \rfloor+1)} |\partial^{\beta}_x f(x)|^2\big)^{r-\lfloor r \rfloor} dx \\
\leq \|\langle x \rangle^{\lfloor r \rfloor} \partial_x^{\beta} f \|^{2(\lfloor r \rfloor +1- r)}_{L^2(\rr^d)} \|\langle x \rangle^{\lfloor r \rfloor+1} \partial_x^{\beta} f \|^{2(r-\lfloor r \rfloor)}_{L^2(\rr^d)},
\end{multline}
where $\lfloor \cdot \rfloor$ denotes the floor function.
Since the above inequality clearly holds for $r \in \nn$, we deduce from \eqref{int} and \eqref{holder0} that for all $r \geq 0$ and $\beta \in \nn^d$,
\begin{align}\label{GS_1}
\|\langle x \rangle^r \partial_x^{\beta} f \|_{L^2(\rr^d)} & \leq C A^{r+|\beta|} (\lfloor r \rfloor!)^{(\lfloor r \rfloor +1-r) \nu} \big((\lfloor r \rfloor+1)!\big)^{(r-\lfloor r \rfloor) \nu} (|\beta|!)^{\mu} \\ \nonumber
& \leq C A^{r+|\beta|} \big((\lfloor r \rfloor+1)!\big)^{\nu} (|\beta|!)^{\mu} \\ \nonumber
& \leq  C A^{r+|\beta|} (\lfloor r \rfloor+1)^{(\lfloor r \rfloor+1)\nu} (|\beta|!)^{\mu} \\ \nonumber
& \leq C A^{r+|\beta|} (r+1)^{(r+1)\nu} (|\beta|!)^{\mu}.
\end{align}
It follows from \eqref{GS_1} that for all $p \in \nn^*$, $\beta \in \nn^d$,
\begin{align}\label{puiss_frac}
\|\langle x \rangle^{\delta p} \partial_x^{\beta} f \|_{L^2(\rr^d)} & \leq C A^{p+|\beta|} (p+1)^{(\delta p+1)\nu} (|\beta|!)^{\mu} \leq C A^{p+|\beta|} (2p)^{(\delta p+1)\nu} (|\beta|!)^{\mu} \\ \notag
& \leq C (2^{\nu}A)^{p+|\beta|} p^{\nu} (2p)^{\delta \nu p} (|\beta|!)^{\mu} \leq C (8^{\nu}e^{\nu}A)^{p+|\beta|} (p!)^{\delta \nu} (|\beta|!)^{\mu},
\end{align}
since for all positive integer $p \geq 1$,
\begin{equation*}
p+1 \leq 2p \leq 2^p \quad \text{and} \quad p^p \leq e^p p!.
\end{equation*}
Notice that from \eqref{int}, since $8^{\nu} e^{\nu} \geq 1$, estimates \eqref{puiss_frac} also hold for $p=0$.  This ends the proof of Lemma~\ref{interpolation}.
\end{proof}

\subsection{Quasi-analytic sequences}\label{qa_section}
This section is devoted to recall some properties of quasi-analytic sequences and to state a multidimensional version of the Nazarov-Sodin-Volberg theorem (Corollary~\ref{NSV}). This theorem plays a key role in the proof of Theorem~\ref{general_uncertaintyprinciple}. We begin by a lemma which provides some quasi-analytic sequences and quantitative estimates on the Bang degree $n_{t, \mathcal{M}, r}$ defined in \eqref{Bang}:

\begin{lemma}\label{ex_qa_sequence}
Let $0<s \leq 1$, $A \geq 1$ and $\mathcal{M}_s= (A^p(p!)^s)_{p \in \nn}$. If $0<s<1$, then for all $0<t \leq 1$, $r>0$,
\begin{equation}
n_{t, \mathcal{M}_s, r} \leq 2^{\frac{1}{1-s}}\big(1-\log t+(Ar)^{\frac{1}{1-s}}\big).
\end{equation}
If $s=1$, then for all $0< t \leq 1$, $r>0$,
\begin{equation}\label{cass1}
n_{t, \mathcal{M}_1, r} \leq (1-\log t) e^{Ar} .
\end{equation}
Moreover, $$\forall 0 < s \leq 1, \forall p \in \nn^*, \quad 0 \leq \gamma_{\mathcal{M}_s}(p) \leq s.$$
\end{lemma}
\medskip
\begin{proof}
Let $0<s\leq1$ and $0< t \leq 1$. The sequence $\mathcal{M}_s$ is logarithmically convex. By using that the Riemann series $$A^{-1}\sum \frac{1}{p^s} = \sum \frac{A^{p-1}((p-1)!)^s}{A^p(p!)^s}$$ is divergent, we notice that for all $r>0$, $n_{t, \mathcal{M}_s,r} < +\infty$. When $0<s<1$, we have that for all integers $p \geq 1$,
\begin{equation*}
\frac{1}{1-s}\big((p+1)^{1-s}-p^{1-s}\big)=\int_{p}^{p+1} \frac{1}{x^s} dx \leq \frac{1}{p^s}.
\end{equation*}
It follows that for all $N \in \nn^*$, 
\begin{equation*}
\frac{1}{1-s} \big((N+1)^{1-s}-(-\log t+1)^{1-s}\big) \leq \sum_{-\log t <p \leq N} \frac{1}{p^s}.
\end{equation*}
By taking $N= n_{t, \mathcal{M}_s, r}$ and since $0< 1-s < 1$, it follows that
\begin{equation*}
n_{t, \mathcal{M}_s, r} \leq \Big((1-\log t)^{1-s}+Ar\Big)^{\frac{1}{1-s}}.
\end{equation*}
The result then follows by using the basic estimate
\begin{equation*}
\forall x, y \geq0, \quad (x+y)^{\frac{1}{1-s}} \leq 2^{\frac{1}{1-s}} \max\Big(x^{\frac{1}{1-s}}, y^{\frac{1}{1-s}}\Big) \leq 2^{\frac{1}{1-s}} \Big(x^{\frac{1}{1-s}}+y^{\frac{1}{1-s}}\Big).
\end{equation*}
 By proceeding in the same manner in the case when $s=1$, we deduce the upper bound \eqref{cass1} thanks to the formula 
\begin{equation*}
\forall p \in \nn^*, \quad \log(p+1)-\log p = \int_p^{p+1} \frac{dx}{x} \leq \frac{1}{p}.
\end{equation*}
By noticing that
\begin{equation*}
\forall 0<s \leq 1, \forall j \in \nn^*, \quad (j+1)^s-j^s= \int_j^{j+1} \frac{s}{x^{1-s}} dx \leq s \frac{1}{j^{1-s}},
\end{equation*}
we finally obtain that for all $0<s \leq 1$, $p \in \nn^*$,
\begin{equation*}
\gamma_{\mathcal{M}_s}(p)= \sup_{1 \leq j \leq p} j \Big(\frac{M_{j+1} M_{j-1}}{M_j^2} -1\Big) = \sup_{1\leq j \leq p} j^{1-s} \big((j+1)^s-j^s\big) \leq s< +\infty.
\end{equation*}
\end{proof}

Let us now prove Proposition~\ref{ex_qa_bertrand}.
This proof uses the following lemmas established in~\cite{AlphonseMartin}:

\begin{lemma}[{\cite[Lemma~4.4]{AlphonseMartin}}]\label{relation} Let $\mathcal M=(M_p)_{p \in \nn}$ and $\mathcal M'=(M'_p)_{p\in\nn}$ be two sequences of positive real numbers satisfying $$\forall p \in \nn, \quad M_p \le M'_p.$$ If $\mathcal M'$ is a quasi-analytic sequence, so is the sequence $\mathcal M$.
\end{lemma}
\medskip
\begin{lemma}[{\cite[Lemma~4.5]{AlphonseMartin}}]\label{linearcomb} Let $\Theta : [0,+\infty)\rightarrow [0,+\infty)$ be a continuous function. If the associated sequence $\mathcal{M}^{\Theta}$ in \eqref{lc_sequence} is quasi-analytic, so is $\mathcal M^{T\Theta+c}$ for all $c \geq 0$ and $T>0$.
\end{lemma}
\medskip
Let $k \geq 1$ be a positive integer, $\frac{1}{2} \leq s \leq 1$ and $\Theta_{k,s} : [0,+\infty) \longrightarrow [0,+\infty)$ be the non-negative function defined in Proposition~\ref{ex_qa_bertrand}.
We first notice that the assumption $\text{(H1)}$ clearly holds for $\mathcal M^{\Theta_{k,s}}$. Let us check that the assumption $\text{(H2)}$ holds as well. To that end, we notice that
\begin{equation*}
\forall t \geq 0, \quad  \Theta_{k,s}(t) \leq t+1
\end{equation*}
and we deduce that
\begin{equation*}
\forall p \in \nn, \quad M^{\Theta_{k,s}}_p \geq \sup_{t \geq 0} t^p e^{-(t+1)}= e^{-1} \Big( \frac{p}{e}\Big)^p. 
\end{equation*}
It remains to check that $\text{(H3)}_s$ holds. Thanks to the morphism property of the logarithm, it is clear that 
\begin{equation*}
\Theta_{k,s}(t) \underset{t \to +\infty}{\sim} s\Theta_{k,1}(t^s)
\end{equation*}
and this readily implies that there exists a positive constant $C_{k,s}>0$ such that
\begin{equation*}
\forall t \geq 0, \quad \Theta_{k,s}(t) + C_{k,s} \geq s\Theta_{k,1}\big(t^s\big).
\end{equation*}
It follows that 
\begin{align*}
\forall p \in \nn, \quad \Big(M^{\Theta_{k,s}}_p\Big)^s=e^{sC_{k,s}}\Big(M^{\Theta_{k,s}+C_{k,s}}_p\Big)^s & \leq e^{sC_{k,s}} \sup_{t \geq 0} t^{sp}e^{-s^2 \Theta_{k,1}(t^s)} \\
&= e^{sC_{k,s}} \sup_{t \geq 0} t^{p}e^{-s^2 \Theta_{k,1}(t)} \\
& = e^{sC_{k,s}} M_p^{s^2\Theta_{k,1}}.
\end{align*}
By using Proposition~\ref{ex_theta1} together with Lemmas \ref{relation} and \ref{linearcomb}, the quasi-analyticity of the sequence $\big((M^{\Theta_{k,s}}_p)^s\big)_{p \in \nn}$ follows from the quasi-analyticity of $\mathcal{M}^{\Theta_{k,1}}$.

The following result by Nazarov, Sodin and Volberg \cite{NSV} provides an uniform control on the uniform norm of quasi-analytic functions ruled by their values on a positive measurable subset. Originally stated in \cite[Theorem~B]{NSV}, it has been used by Jaye and Mitkovski (\cite{JayeMitkovski}) in the following form: 
\begin{theorem}[{\cite[Theorem~2.5]{JayeMitkovski}}]\label{JayeNSV} Let $\mathcal{M}=(M_p)_{p \in \nn}$ be a logarithmically convex quasi-analytic sequence with $M_0=1$ and $f \in \mathcal{C}_{\mathcal{M}}([0,1]) \setminus \{0\}$. For any interval $I \subset [0,1]$ and measurable subset $\mathcal{J} \subset I$ with $|\mathcal{J}| >0$,
\begin{equation*}
\sup_{I} |f| \leq \Big(\frac{\Gamma_{\mathcal{M}}(2n_{\|f\|_{L^{\infty}([0,1])}, \mathcal{M},e}) |I|}{|\mathcal{J}|} \Big)^{2n_{\|f\|_{L^{\infty}([0,1])}, \mathcal{M},e}} \sup_{\mathcal{J}} |f|.
\end{equation*}
\end{theorem}
The following corollary is instrumental in this work:
\medskip
\begin{corollary}\label{NSV} Let $\mathcal{M}=(M_p)_{p \in \nn}$ be a logarithmically convex quasi-analytic sequence with $M_0=1$ and $0< s, t \leq 1$. There exists a positive constant $C=C(\mathcal{M}) \geq 1$ such that  for any interval $I \subset [0,1]$ and measurable subset $\mathcal{J} \subset I$ with $|\mathcal{J}| \geq s >0$,
\begin{equation*}
\forall f \in \mathcal{C}_{\mathcal{M}}([0,1]) \text{ with } \|f\|_{L^{\infty}([0,1])} \geq t, \quad \sup_{I} |f| \leq \Big(\frac{\Gamma_{\mathcal{M}}(2n_{t, \mathcal{M},e}) |I|}{s} \Big)^{2n_{t, \mathcal{M},e}} \sup_{\mathcal{J}} |f|.
\end{equation*}
\end{corollary}
\medskip
Corollary~\ref{NSV} is directly deduced from Theorem~\ref{JayeNSV} by noticing that for all $f \in \mathcal{C}_{\mathcal{M}}([0,1])$ satisfying $\| f \|_{L^{\infty}([0,1])} \geq t$, $$n_{\|f\|_{L^{\infty}([0,1])}, \mathcal{M},e} \leq n_{t, \mathcal{M},e}.$$
In order to use this result in control theory, we need a multidimensional version of Corollary~\ref{NSV}:
\medskip

 \begin{proposition}\label{NSV_multid}
Let $d \geq 1$ and $U$ be a non-empty bounded open convex subset of $\rr^d$ satisfying $|\partial U|=0$. Let $\mathcal{M}=(M_p)_{p \in \nn}$ be a logarithmically convex quasi-analytic sequence with $M_0=1$, $0< \gamma \leq 1$ and $0<t\leq1$. For any measurable subset $E \subset U$ satisfying $|E| \geq \gamma |U|>0$, we have
\begin{multline}\label{NSV_estimate}
\forall f \in \mathcal{C}_{\mathcal{M}}(U) \text{ with } \|f\|_{L^{\infty}(U)} \geq t, \\
 \sup_{U} |f| \leq \Big(\frac{d}{\gamma} \Gamma_{\mathcal{M}}\big(2n_{t, \mathcal{M},d \diam(U) e}\big)\Big)^{2n_{t, \mathcal{M},d \diam(U) e}} \sup_{E} |f|.
\end{multline}
\end{proposition}
\medskip

\begin{proof}
Let $0< \gamma \leq 1$ and $0<t \leq 1$. Let $E$ be a measurable subset of $U$ satisfying $|E| \geq \gamma |U|>0$ and $f \in \mathcal{C}_{\mathcal{M}}(U)$ with $\|f\|_{L^{\infty}(U)} \geq t$. Since $\overline{U}$ is compact and that $f$ can be extended as a continuous map on $\overline{U}$, there exists $x_0 \in \overline{U}$ such that 
\begin{equation}\label{max}
\sup_U |f|= |f(x_0)|.
\end{equation}
By using spherical coordinates, we have
\begin{equation*}
|E|= \int_{\rr^d} \un_{E}(x) dx =  \int_{\rr^d} \un_{E}(x_0+x) dx
= \int_0^{+\infty} \int_{\mathbb{S}^{d-1}} \un_{E} (x_0 + t \sigma) d\sigma t^{d-1}dt.
\end{equation*}
Since $\overline{U}$ is convex, we deduce that
 \begin{align*}\label{m1}
0<|E| & =\int_{\mathbb{S}^{d-1}} \int_{0}^{J_{\overline{U}}(\sigma)} \un_{E} (x_0 + t \sigma) t^{d-1}dt d\sigma \\
 & = \int_{\mathbb{S}^{d-1}} J_{\overline{U}}(\sigma)^d\int_{0}^{1} \un_{E} \big(x_0 + J_{\overline{U}}(\sigma) t \sigma\big)  t^{d-1}dt d\sigma \\ 
& \leq \int_{\mathbb{S}^{d-1}} J_{\overline{U}}(\sigma)^d\int_{0}^{1} \un_{E} \big(x_0 + J_{\overline{U}}(\sigma) t \sigma\big)dt d\sigma \leq \int_{\mathbb{S}^{d-1}} J_{\overline{U}}(\sigma)^d |I_{\sigma}| d\sigma,
\end{align*}
with
\begin{equation}\label{jauge}
 J_{\overline{U}}(\sigma) =\sup\{t\geq0: \, x_0+ t\sigma \in \overline{U} \} \quad \text{and} \quad I_{\sigma} = \Big\{t \in [0,1]: \, x_0 + J_{\overline{U}}(\sigma) t \sigma \in E \Big\},
\end{equation}
when $\sigma \in \mathbb{S}^{d-1}$. Notice that $$\forall \sigma \in \mathbb{S}^{d-1}, \quad J_{\overline{U}}(\sigma) <+\infty,$$
since $\overline{U}$ is bounded.
It follows that there exists $\sigma_0 \in \mathbb{S}^{d-1}$ such that 
\begin{equation}\label{m3}
|E| \leq |I_{\sigma_0}| \int_{\mathbb{S}^{d-1}} J_{\overline{U}}(\sigma)^d d\sigma .
\end{equation}
By using the assumption that $|\partial U|=0$ and $U$ is an open set, we observe that
\begin{equation*}
|U|=|\overline{U}|= \int_{\mathbb{S}^{d-1}} J_{\overline{U}}(\sigma)^d \int_0^1 t^{d-1}dt d\sigma = \frac{1}{d} \int_{\mathbb{S}^{d-1}} J_{\overline{U}}(\sigma)^d d\sigma.
\end{equation*}
By using that $|E| \geq \gamma |U|$, the estimate \eqref{m3} and the above formula provide the lower bound
\begin{equation} \label{m4}
|I_{\sigma_0}| \geq \frac{\gamma}{d}>0.
\end{equation}
Setting
\begin{equation}\label{fonct_aux}
\forall t \in [0,1], \quad g(t)=f\big(x_0 + J_{\overline{U}}(\sigma_0) t \sigma_0 \big),
\end{equation}
we notice that this function is well-defined as $x_0 + J_{\overline{U}}(\sigma_0) t \sigma_0 \in \overline{U}$ for all $t \in [0,1]$. We deduce from the fact that $f \in \mathcal{C}_{\mathcal{M}}(U)$, the estimate $$ J_{\overline{U}}(\sigma_0) \leq \diam(\overline{U})=\diam(U),$$
where $\diam(U)$ denotes the Euclidean diameter of $U$, and the multinomial formula that for all $p \in \nn$,
\begin{multline*}
\|g^{(p)}\|_{L^{\infty}([0,1])} \leq \sum_{\substack{\beta \in \nn^d, \\ |\beta|=p}} \frac{p!}{\beta !} \|\partial^{\beta}_x f \|_{L^{\infty}(\overline{U})} \big(J_{\overline{U}}(\sigma_0)\big)^p \\ \leq \bigg(  \sum_{\substack{\beta \in \nn^d, \\ |\beta|=p}} \frac{p!}{\beta !}\bigg) \diam(U)^p M_p = \big(d\diam(U)\big)^p M_p.
\end{multline*}
We observe that the new sequence $$\mathcal{M}':= \Big(\big(d\diam(U)\big)^p M_p \Big)_{p \in \nn},$$ inherits from $\mathcal{M}$ its logarithmical convexity, its quasi-analytic property with the following identity for the associated Bang degrees
\begin{equation*}
n_{t,\mathcal{M}',e}= n_{t,\mathcal{M},d\diam(U)e}.
\end{equation*}
The function $g$ belongs to $\mathcal{C}_{\mathcal{M}'}([0,1])$. By using from \eqref{m4} that $|I_{\sigma_0}| >0$ and $\|g\|_{L^{\infty}([0,1])} \geq |g(0)|=\|f\|_{L^{\infty}(U)} \geq t$, we can apply Corollary~\ref{NSV} to obtain that
\begin{equation}\label{NSV_1}
\sup_{[0,1]} |g| \leq \Big(\frac{\Gamma_{\mathcal{M}'}(2n_{t, \mathcal{M}',e})}{|I_{\sigma_0}|} \Big)^{2n_{t, \mathcal{M}',e}} \sup_{I_{\sigma_0}} |g|.
\end{equation}
By noticing that $$\Gamma_{\mathcal{M}}=\Gamma_{\mathcal{M}'},$$
we deduce from \eqref{max}, \eqref{jauge}, \eqref{m4}, \eqref{fonct_aux} and \eqref{NSV_1} that
\begin{multline*}
\sup_U |f| = |f(x_0)| =|g(0)| \leq \sup \limits_{[0,1]} |g| \leq \Big(\frac{d}{\gamma} \Gamma_{\mathcal{M}}(2n_{t, \mathcal{M},d\diam(U) e}) \Big)^{2n_{t, \mathcal{M},d\diam(U) e}} \sup_{I_{\sigma_0}} |g| \\ 
\leq \Big(\frac{d}{\gamma} \Gamma_{\mathcal{M}}(2n_{t, \mathcal{M},d\diam(U) e}) \Big)^{2n_{t, \mathcal{M},d\diam(U) e}} \sup_E |f|.
\end{multline*}
This ends the proof of Proposition~\ref{NSV_multid}.
\end{proof}

In order to use estimates as (\ref{NSV_estimate}) to derive the null-controllability of evolution equations posed in $L^2(\rr^d)$, we need the following $L^2$-version of the Nazarov-Sodin-Volberg Theorem:

\medskip

\begin{proposition}\label{NSV_multid_L2}
Let $d \geq 1$ and $U$ be a non-empty bounded open convex subset of $\rr^d$.
Let $\mathcal{M}=(M_p)_{p \in \nn}$ be a logarithmically convex quasi-analytic sequence with $M_0=1$, $0< \gamma \leq 1$ and $0<t \leq 1$. If $E \subset U$ is a measurable subset satisfying $|E| \geq \gamma |U|$, then for all $f \in \mathcal{C}_{\mathcal{M}}(U)$ with $\|f\|_{L^{\infty}(U)} \geq t$,
\begin{equation*}
 \int_U |f(x)|^2 dx \leq \frac{2}{\gamma}\Big(\frac{2d}{\gamma} \Gamma_{\mathcal{M}}\big(2n_{t, \mathcal{M},d\diam(U) e}\big)\Big)^{4n_{t, \mathcal{M},d\diam(U) e}}  \int_E |f(x)|^2 dx.
\end{equation*}
\end{proposition}
\medskip

\begin{proof}
Let $0<t\leq 1$, $f \in \mathcal{C}_{\mathcal{M}}(U)$ so that $\|f\|_{L^{\infty}(U)} \geq t$ and $E$ be a subset of $U$ satisfying $|E| \geq \gamma |U| >0$. Setting
\begin{equation*}
\tilde{E}= \Big\{x \in E: \  |f(x)|^2 \leq \frac{2}{|E|} \int_E |f(y)|^2 dy\Big\},
\end{equation*}
we observe that
\begin{equation}\label{m_20}
\int_E |f(x)|^2dx \geq \int_{ E \setminus \tilde{E}} |f(x)|^2 dx \geq \frac{2|E \setminus \tilde{E}|}{|E|} \int_E |f(x)|^2dx.
\end{equation}
Let us prove by contradiction that the integral
\begin{equation*}
\int_E |f(x)|^2 dx >0,
\end{equation*}
is positive.
If $$\int_E |f(x)|^2 dx =0,$$ then, $$E_{\mathcal{Z}}=\Big\{ x \in E: \quad f(x)=0 \Big\},$$ satisfies $|E_{\mathcal{Z}}|=|E|>0$. We therefore deduce from Proposition~\ref{NSV_multid}, since $\|f\|_{L^{\infty}(U)} \geq t$ and $|E_{\mathcal{Z}}|>0$, that $f=0$ on $U$. This contradicts the assumption $\| f \|_{L^{\infty}(U)} \geq t>0$ and therefore 
\begin{equation*}
\int_E |f(x)|^2 dx >0.
\end{equation*}
We deduce from \eqref{m_20} that
\begin{equation*}\label{m_21}
|\tilde{E}| = |E|-|E\setminus \tilde{E}| \geq \frac{|E|}{2} \geq \frac{\gamma}{2} |U| >0.
\end{equation*}
Applying Proposition~\ref{NSV_multid} provides that
\begin{multline*}
\sup_U |f| \leq \Big(\frac{2d}{\gamma} \Gamma_{\mathcal{M}}(2n_{t, \mathcal{M},d\diam(U) e}) \Big)^{2n_{t, \mathcal{M},d\diam(U) e}} \sup_{\tilde{E}} |f| \\ 
\leq \Big(\frac{2d}{\gamma} \Gamma_{\mathcal{M}}(2n_{t, \mathcal{M},d\diam(U) e})\Big)^{2n_{t, \mathcal{M},d\diam(U) e}} \frac{\sqrt{2}}{\sqrt{|E|}} \Big(\int_E|f(x)|^2dx\Big)^{\frac{1}{2}}.
\end{multline*}
It follows that
\begin{align*}
\int_U |f(x)|^2dx \leq |U|\big(\sup_U |f|\big)^2 & \leq  \Big(\frac{2d}{\gamma} \Gamma_{\mathcal{M}}(2n_{t, \mathcal{M},d\diam(U) e}) \Big)^{4n_{t, \mathcal{M},d\diam(U) e}}\frac{2|U|}{|E|}  \int_E |f(x)|^2dx \\
& \leq \Big(\frac{2d}{\gamma} \Gamma_{\mathcal{M}}(2n_{t, \mathcal{M},d\diam(U) e}) \Big)^{4n_{t, \mathcal{M},d\diam(U) e}} \frac{2}{\gamma}\int_E |f(x)|^2dx.
\end{align*}
This concludes the proof of Proposition~\ref{NSV_multid_L2}.
\end{proof}
In \cite{JayeMitkovski}, the authors also establish a multi-dimensional version and a $L^2$-version of the Nazarov-Sodin-Volberg Theorem (Theorem~\ref{JayeNSV}) but the constants obtained there are less explicit than the ones given in Propositions~\ref{NSV_multid} and \ref{NSV_multid_L2}. Quantitative constants will be essential in Section~\ref{null_controllability_results} to set up an adapted Lebeau-Robbiano method in order to derive null-controllability results.
We end this section by illustrating the above result with an example:
\begin{example}\label{NSV_example}
Let $0< s \leq 1$, $A\geq 1$, $R>0$, $d \geq 1$, $0<t \leq 1$, $0< \gamma \leq 1$ and $\mathcal{M}=(A^p (p!)^s)_{p \in \nn}$. Let $E \subset B(0,R)$ be a measurable subset of the Euclidean ball centered at $0$ with radius $R$ such that $|E| \geq \gamma |B(0,R)|$. There exists a constant $K=K(s, d) \geq 1$ such that for all $f \in \mathcal{C}_{\mathcal{M}}(B(0,R))$ with $\|f\|_{L^{\infty}(B(0,R))} \geq t$,
\begin{equation*}
\| f \|_{L^{\infty}(B(0,R))} \leq C_{t, A, s, R, \gamma, d} \|f\|_{L^{\infty}(E)} \quad \text{and} \quad \| f \|_{L^2(B(0,R))} \leq C_{t, A, s, R, \gamma, d} \|f\|_{L^2(E)} ,
\end{equation*}
where when $0<s<1$,
$$0<C_{t, A, s, R, \gamma, d} \leq \Big(\frac{K}{\gamma}\Big)^{K(1-\log t+ (AR)^{\frac{1}{1-s}})}$$
and when $s=1$,
$$0<C_{t, A, 1, R, \gamma, d} \leq \Big(\frac{K}{\gamma}\Big)^{K(1-\log t)e^{KAR}}.$$
\end{example}

Let us check that Example~\ref{NSV_example} is a consequence of Propositions~\ref{NSV_multid} and \ref{NSV_multid_L2}, together with Lemma~\ref{ex_qa_sequence}. We deduce from Propositions~\ref{NSV_multid} and \ref{NSV_multid_L2} that for all $f \in \mathcal{C}_{\mathcal{M}}(B(0,R))$ with $\|f\|_{L^{\infty}(B(0,R))} \geq t$,
\begin{equation*}
\| f \|_{L^{\infty}(B(0,R))} \leq \Big(\frac{d}{\gamma} \Gamma_{\mathcal{M}}(2n_{t, \mathcal{M},2Rd e}) \Big)^{2n_{t, \mathcal{M},2Rd e}} \|f\|_{L^{\infty}(E)}
\end{equation*}
and
\begin{equation*}
\| f \|_{L^2(B(0,R))} \leq \sqrt{\frac{2}{\gamma}} \Big(\frac{2d}{\gamma} \Gamma_{\mathcal{M}}(2n_{t, \mathcal{M},2Rd e}) \Big)^{2n_{t, \mathcal{M},2Rd e}} \|f\|_{L^2(E)}.
\end{equation*}
Furthermore, Lemma~\ref{ex_qa_sequence} provides that $$\forall n \in \nn^*, \quad \Gamma_{\mathcal{M}}(n) \leq e^{4+4s}$$ and if $0<s<1$ then,
\begin{equation*}
n_{t, \mathcal{M}, 2Rde} \leq 2^{\frac{1}{1-s}}\big(1-\log t+(2ARde)^{\frac{1}{1-s}}\big),
\end{equation*}
whereas if $s=1$, then
\begin{equation*}
n_{t, \mathcal{M}, 2Rde} \leq (1-\log t) e^{2ARde}. 
\end{equation*}
The result of Example~\ref{NSV_example} therefore follows from the above estimates.
\subsection{Slowly varying metrics}{\label{vsm}}
This section is devoted to recall basic facts about slowly varying metrics. We refer the reader to~\cite{Hormander} (Section 1.4) for the proofs of the following results.
Let $X$ be an open subset in a finite dimensional $\rr$-vector space $V$ and $\|\cdot\|_x$ a norm in $V$ depending on $x \in X$. The family of norms $(\|\cdot\|_x)_{x \in X}$ is said to define a slowly varying metric in $X$ if there exists a positive constant $C \geq 1$ such that for all $x \in X$ and for all $y \in V$ satisfying $\|y-x\|_x <1$, then $y \in X$ and 
\begin{equation}{\label{equiv}}
\forall v \in V, \quad \frac{1}{C} \|v \|_x \leq \|v\|_y \leq C \|v \|_x.
\end{equation}

\medskip
\begin{lemma}\label{slowmet}\cite[Example~1.4.8]{Hormander}.
Let $X$ be an open subset in a finite dimensional $\rr$-vector space $V$ and $d(x)$ a $\frac{1}{2}$-Lipschitz continuous function, positive in $X$ and zero in $V \setminus X$, satisfying 
\begin{equation*}
\forall x,y \in X, \quad |d(x) - d(y) | \leq \frac{1}{2}\|x-y \|,
\end{equation*}
where $\|\cdot\|$ is a fixed norm in $V$. Then, the family of norms $(\|\cdot\|_x)_{x \in X}$ given by
\begin{equation*}\label{family_norms}
\|v\|_x= \frac{ \|v\|}{d(x)}, \quad x \in X, v \in V,
\end{equation*}
defines a slowly varying metric in X.
\end{lemma}

\medskip

The proof given in \cite[Example~1.4.8]{Hormander} shows more generally that result of Lemma~\ref{slowmet} holds true as well when $d$ is a contraction mapping function, that is, when there exists $0 \leq k <1$ such that
\begin{equation*}
\forall x,y \in X, \quad |d(x) - d(y) | \leq k \|x-y \|.
\end{equation*}
Let us consider the case when $X=V=\rr^d$ and $\|\cdot\|$ is the Euclidian norm. If $0 < \delta \leq 1$ and $0< R <\frac{1}{\delta}$, then the gradient of the function $\rho_\delta(x)=R\left\langle x\right\rangle^{\delta}$ given by 
$$\forall x \in \rr^d, \quad \nabla \rho_\delta(x)=R \delta \frac{x}{\left\langle x\right\rangle^{2-\delta}},$$ 
satisfies $\| \nabla \rho_\delta\|_{L^{\infty}(\rr^d)} \leq R \delta <1$. The mapping $\rho_{\delta}$ is then a positive contraction mapping and Lemma~\ref{slowmet} shows that the family of norms  $\|\cdot\|_x= \frac{\|\cdot\|}{R \left\langle x\right\rangle^{\delta}}$ defines a slowly varying metric on $\rr^d$.

\medskip

\begin{theorem}{\label{slowmetric}}
\cite[Theorem~1.4.10]{Hormander}.
Let $X$ be an open subset in $V$ a $\rr$-vector space of finite dimension $d \geq 1$ and $(\|\cdot\|_x)_{x \in X}$ be a family of norms in $V$ defining a slowly varying metric. Then, there exists a sequence $(x_k)_{k \geq 0} \in X^{\nn}$ such that the balls
\begin{equation*}
B_k=\left\{x \in V:\ \|x-x_k \|_{x_k} <1 \right\} \subset X,
\end{equation*}
form a covering of $X$, 
$$X = \bigcup \limits_{k=0}^{+\infty} B_k,$$ 
such that the intersection of more than $N=\big(4 C^3+1 \big)^d$ two by two distinct balls $B_k$ is always empty, where $C \geq 1$ denotes the positive constant appearing in the slowness condition \emph{(\ref{equiv})}.
\end{theorem}

\end{document}